\documentclass[reqno,11pt]{amsart}
\usepackage{amsthm,amsfonts,amssymb,euscript,
mathrsfs,graphics,color,amsmath,latexsym,marginnote}
 
\theoremstyle{plain}

\usepackage{hyperref}
\usepackage{times}

\numberwithin{equation}{section}

\newtheorem{theorem}{Theorem}[section]
\newtheorem{proposition}[theorem]{Proposition}
\newtheorem{lemma}[theorem]{Lemma}
\newtheorem{corollary}[theorem]{Corollary}

\newtheorem{remark}[theorem]{Remark}
\newtheorem{remarks}[theorem]{Remark}

\newtheorem{definition}[theorem]{Definition}
\newcommand{\be}{\begin{equation}}
\newcommand{\ee}{\end{equation}}
\newcommand{\om}{\omega}
\newcommand{\eps}{\varepsilon}
\newcommand{\e}{\varepsilon}

\newcommand{\al}{\alpha}

\newcommand{\R}{\mathbb R}
\newcommand{\C}{\mathbb C}

\newcommand{\Z}{\mathbb Z}

\newcommand{\N}{\mathbb N}
\newcommand{\T}{\mathbb T}

\newcommand{\ka}{\kappa}
\newcommand{\s }{\sigma }
\newcommand{\ii }{{\rm i} }

\newcommand{\g }{\gamma}

\newcommand{\M}{\mathcal M}

\newcommand{\vphi}{\varphi }

\renewcommand{\t }{\tau }
\renewcommand{\o }{\omega }
\renewcommand{\O }{\mathcal{O} }

\newcommand{\bral}{[ \! [} 
\newcommand{\brar}{] \! ]}

\renewcommand{\L}{\mathcal{L}}

\newcommand{\A}{\mathcal{A}}
\newcommand{\Nc}{\mathcal{N}}
\newcommand{\mR}{\mathcal{R}}

\newcommand{\pa}{\partial}

\newcommand{\Rc}{\mathcal R}

\renewcommand{\S}{\mathbb{S}}

\def\ba{\begin{aligned}}
\def\ea{\end{aligned}}
\def\beginm{\begin{multline}}
\def\endm{\end{multline}}

\makeatletter

\def\l@subsection{\@tocline{2}{0pt}{2.5pc}{5pc}{}}
\def\l@subsubsection{\@tocline{3}{0pt}{4.5pc}{5pc}{}}
\renewcommand\tocchapter[3]{%
  \indentlabel{\@ifnotempty{#2}{\ignorespaces#2.\quad}}#3%
}
\newcommand\@dotsep{4.5}
\def\@tocline#1#2#3#4#5#6#7{\relax
  \ifnum #1>\c@tocdepth 
  \else
    \par \addpenalty\@secpenalty\addvspace{#2}%
    \begingroup \hyphenpenalty\@M
    \@ifempty{#4}{%
      \@tempdima\csname r@tocindent\number#1\endcsname\relax
    }{%
      \@tempdima#4\relax
    }%
    \parindent\z@ \leftskip#3\relax \advance\leftskip\@tempdima\relax
    \rightskip\@pnumwidth plus1em \parfillskip-\@pnumwidth
    #5\leavevmode\hskip-\@tempdima{#6}\nobreak
    \leaders\hbox{$\m@th\mkern \@dotsep mu\hbox{.}\mkern \@dotsep mu$}\hfill
    \nobreak
    \hbox to\@pnumwidth{\@tocpagenum{#7}}\par
    \nobreak
    \endgroup
  \fi}
\makeatother
\AtBeginDocument{%
\makeatletter
\expandafter\renewcommand\csname r@tocindent0\endcsname{0pt}
\makeatother
}
\def\l@subsection{\@tocline{2}{0pt}{2.5pc}{5pc}{}}

\begin{document}

\title[Reducibility of Schr\"odinger equation on a Zoll Manifold]{Reducibility of  Schr\"odinger equation 
on a Zoll Manifold with unbounded potential}

\date{}

\author{Roberto Feola}
\address{\scriptsize{Laboratoire de Math\'ematiques Jean Leray, Universit\'e 
de Nantes, 
UMR CNRS 6629
\\
2, rue de la Houssini\`ere 
\\
44322 Nantes Cedex 03, France}
}
\email{roberto.feola@univ-nantes.fr}

\author{Beno\^it Gr\'ebert}
\address{\scriptsize{Laboratoire de Math\'ematiques Jean Leray, Universit\'e de Nantes, UMR CNRS 6629\\
2, rue de la Houssini\`ere \\
44322 Nantes Cedex 03, France}}
\email{benoit.grebert@univ-nantes.fr}

\author{Trung Nguyen}
\address{\scriptsize{Laboratoire de Math\'ematiques Jean Leray, Universit\'e 
de Nantes, 
UMR CNRS 6629
\\
2, rue de la Houssini\`ere 
\\
44322 Nantes Cedex 03, France}
}
\email{trungisp58@gmail.com}

\thanks{During the preparation of this work 
the three authors benefited from the support 
of the Centre Henri Lebesgue ANR-11-LABX- 0020-01 and  of 
ANR -15-CE40-0001-02  ``BEKAM''  of the Agence 
Nationale de la Recherche.  
R. F. was also supported by ERC starting grant FAFArE of the European Commission and
 B.G. 
 by  ANR-16-CE40-0013 ``ISDEEC'' of the Agence Nationale de la Recherche.}

\begin{abstract} 
In this article we prove a reducibility result for the linear Schr\"odinger equation on a Zoll manifold with  quasi-periodic in time pseudo-differential perturbation of order less or equal than $1/2$. As far as we know, this is  the first reducibility results for an unbounded perturbation of a linear system which is not integrable.  
\end{abstract}  
  
\maketitle

\tableofcontents

\section{Introduction}
In this article we are interested in the problem of reducibility for the linear Schr\"odinger equation on a Zoll manifold with  quasi-periodic in time pseudo-differential perturbation of order less or equal than $1/2$.\\
We first recall that  a Zoll manifold of dimension $n\in \mathbb{N}$ 
is a 
compact Riemannian manifold $(\mathtt{M}^{n}, g)$ 
such that all the geodesic curves have all the same period $T$. In this paper we assume $T:=2\pi$. 
For example the $n$-dimensional sphere $\S^n$ is 
a Zoll manifold.  
We denote by $\Delta_g$  
the positive Laplace-Beltrami operator on 
$(\mathtt{M}^{n},g)$   and we define $H^{s}({\mathtt{M}^{n}}):={\rm dom}(\sqrt{1+\Delta_g})^{s}$
with $s\in \mathbb{R}$ the usual scale of Sobolev spaces. \\
We denote by $S_{\rm cl}^m(\mathtt{M}^{n})$ 
the space of classical real valued symbols of order $m\in \R$ 
on the cotangent bundle $T^*(\mathtt{M}^{n})$  and we define  $\mathcal{A}_{m}$ 
the associated class of pseudo-differential operators
(see for instance
H\"ormander \cite{ho} for a definition of pseudo-differential operators on a manifold see also  \cite{BGMR2} in the case of a Zoll manifold).\\
We consider the following linear Schr\"odinger 
\[\label{nls}\tag{LS}
\ii \pa_{t}u=\Delta_g u+\eps W(\om t)u\,,\qquad u=u(t,x)\,,
\quad t\in \mathbb{R}\,,\quad x\in \mathtt{M}^{n}\,, 
\]
where $\e>0$ is a small parameter and
$W(\om t)$ is a time dependent unbounded operator 
from $H^{s}(\mathtt{M}^{n})\to H^{s-\delta}(\mathtt{M}^{n})$
for some $\delta\leq 1/2$.
More precisely we assume that $W\in C^\infty(\T^d,\A_\delta)$ with $\delta\leq 1/2$, $d\geq1$. So
the potential 
$t\mapsto W(\om t)$ 
depends on time quasi-periodically with frequency vector $\omega\in \mathbb{R}^{d}$
and   for any $\vphi\in \mathbb{T}^{d}:=(\mathbb{R}/
2\pi\mathbb{Z})^{d}$, 
the linear operator 
$W(\vphi)$ is a pseudo-differential operator of order $\delta$, i.e. 
belongs to $\A_\delta$.

The purpose of this article is to construct a change of variables that transforms the non-autonomous equation \eqref{nls} into an autonomous equation.\\
Our main result is the following.
\begin{theorem}\label{thm:main}  Let $0<\alpha<1$ and $\delta\in \mathbb{R}$, $\delta\leq1/2$. 
Assume that the map
$\vphi\mapsto W(\vphi,\cdot)\in 
\mathcal{A}_{\delta}$ 
is $C^{\infty}$ in  $\vphi\in \mathbb{T}^{d}$.
Then for any $s\in \mathbb{R}$, $s>n/2$ 
 there exists  $\eps_0>0$ and $C>0$ such that, for any $0<\eps\leq \eps_0$ there is
 a set $\mathcal{G}_\eps\subset[1/2,3/2]^{d}\subset\mathbb{R}^{d}$
 with 
 \begin{equation}\label{measure}
 {\rm meas}([1/2,3/2]^{d}\setminus\mathcal{O}_\eps)\leq C\eps^\alpha
 \end{equation}
 such that the following holds.
For any $\omega\in \mathcal{O}_\e$
 there exists 
 a family of linear isomorphisms 
 $\Psi(\vphi)\in \mathcal{L}(H^{s}(\mathtt{M}^{n};\mathbb{C}))$
 and a Hermitian operator 
 $Z\in \A_\delta$ commuting with the Laplacian\footnote{actually $Z$ and $\Delta_g$ can be diagonalized in the same basis of $L^2(\mathtt M^n)$. } and
satisfying 
\begin{equation}\label{stimaZ} \|Z\|_{\mathcal{L}(H^{s}(\mathtt{M}^{n}), H^{s-\delta}(\mathtt{M}^{n}))}\leq C\e\,.
\end{equation}
Furthermore

 $\bullet$
 $\Psi(\vphi) $ is unitary on $L^{2}(\mathtt{M}^{n})$;
 
 $\bullet$
  for any $\frac{n}{2}< s'\leq s$ and any $\omega\in \mathcal{O}_\e$
 \begin{equation}\label{stima}
 \begin{aligned}
 \|\Psi(\vphi)-{\rm Id}&\|_{\mathcal{L}(H^{s'}(\mathtt{M}^{n}),H^{s'-\delta}(\mathtt{M}^{n}))}\\&+
  \|\Psi(\vphi)^{-1}-{\rm Id}\|_{\mathcal{L}(H^{s'}(\mathtt{M}^{n}),H^{s'-\delta}(\mathtt{M}^{n}))}
 & \leq C\e^{1-\alpha}\,,\\
   \|\Psi(\vphi)&\|_{\mathcal{L}(H^{s'}(\mathtt{M}^{n}))}+
  \|\Psi(\vphi)^{-1}\|_{\mathcal{L}(H^{s'}(\mathtt{M}^{n}))}
  &\leq1+ C\e^{1-\alpha}\,,
  \end{aligned}
 \end{equation}
 
 $\bullet$  for any $\frac{n}{2}< s'\leq s$ and any $\omega\in \mathcal{O}_\e$
 the map $t\mapsto u(t,\cdot)\in H^{s'}(\mathtt{M}^{n};\mathbb{C})$
 solves \eqref{nls} if and only if the map
 $t\mapsto v(t,\cdot):=\Psi(\omega t)u(t,\cdot)$ solves the autonomous 
 equation
 \begin{equation}\label{NLSred}
 \ii \pa_t v=\Delta_g v+\eps Z(v)\,.
 \end{equation}
\end{theorem}

As a consequence of our reducibility result, we get a control of the flow generated by the \eqref{nls} equation in the scale of Sobolev spaces:
 \begin{corollary}\label{coro1.3}
 Let $W\in C^\infty(\T^d,\A_\delta)$ with $\delta\leq 1/2$. 
 Then for any $s\in \mathbb{R}$, $s>n/2$ 
 there exists  $\eps_0>0$ and $C>0$ such that, for any $0<\eps\leq \eps_0$ there is
 a set $\mathcal{O}_\eps\subset[1/2,3/2]^{d}\subset\mathbb{R}^{d}$
 satisfying \eqref{measure} such that for any $\om\in \mathcal{O}_\eps$ the flow generated by the \eqref{nls} equation is bounded in $H^s(\mathtt{M}^{n};\mathbb{C})$. \\
 More precisely if
 $u_{0}\in H^{s}(\mathbb{S}^{2};\mathbb{C})$ then
 there exists a unique solution $u \in {C}^{1}\big(\R\,;\,H^{s}\big)$ 
 of \eqref{nls} such that $u(0)=u_{0}$. 
 Moreover, $u$ is almost-periodic in time and satisfies
  \begin{equation}\label{16}
  (1-\eps C)\|u_{0}\|_{H^{s}}\leq \|u(t)\|_{H^{s}}\leq  (1+\eps C)\|u_{0}\|_{H^{s}}, \quad \forall \,t\in \R,
  \end{equation}
  for some $C=C(s)>0$.
  \end{corollary}
Following  the pioneering work \cite{BBM14} 
we prove Theorem \ref{thm:main} in two steps: 
\begin{itemize}
\item {\it The regularization step} where we use the pseudo-differential calculus (and in particular the technics developed in \cite{BGMR2})  to transform equation \eqref{nls} in a system with a smoothing perturbation, still depending on time;
\item {\it The KAM step} where we use a KAM procedure (going back to \cite{Kuk93} but using recent development in \cite{BBHM}) on infinite dimensional matrices to eliminate the time in the new system.
\end{itemize}
The same strategy was recently successfully applied in \cite{BLM18}  to prove the reducibility of non-resonant transport equation on the torus $\T^n$. Our main contribution consists in merging these two recent technics in the context of linear Schr\"odinger equation on Zoll manifold which, in contrast to the transport equation on the torus,   is not an integrable system.

\bigskip

The study of the reducibility problem for Schr\"odinger equations 
 with quasi-periodic in time perturbation has been 
 intensively studied in recent years. 
 The first results adapting the KAM technics 
 were due to Kuksin \cite{Kuk93} 
 followed by many results in one dimensional context  (see in particular \cite{BG, LY, Gtom}).
  More recently the technics 
were adapted to the higher dimensional case \cite{EK09, EGK, GP}. 
To consider unbounded perturbations, a new strategy has been developed in \cite{BBM14, KdVAut} using  the pseudo-differential calculus. 
 Without trying to be exhaustive 
we quote also \cite{FP1, BM1, BBHM, FGP} regarding KAM theory for 
quasi-linear PDEs in one space dimension.
This technics were successfully applied for reducibility problems in various case. 
For one dimensional linear equations with unbounded potential we quote
\cite{Bam17, Bam18, BM18, FGP1}.
In higher space dimensions
we refer to \cite{EK09, GP}
for bounded potential, and to
\cite{BGMR18, Mon18, FGMP, BLM18} for the unbounded cases.

\medskip

\noindent
{\bf Scheme of the proof}\\
As said above the proof consists in a regularization step (section \ref{sec3}) and a KAM step (section \ref{sec4}). In section \ref{sec5} we merge the two procedure to prove Theorem \ref{thm:main}.

\medskip

In the regularization step we prove that we can transform (by using a symplectic map: $u=\Phi (v)$) the original  Schr\"odinger equation \eqref{nls} in a new one 
\begin{equation}\label{nls2}
\ii \pa_{t}v=\Delta_g v+\eps(Z+ R(\om t))v\,, 
\end{equation}
where $Z$ is a pseudo-differential operator of order $\delta$ independent on time and commuting with $\Delta_g$  and $R$ is a $\rho$-regularizing operator in $\L(H^s,H^{s+\rho})$ with $\rho$ arbitrary large. It is based on a normal form procedure developed in \cite{BGMR2}. The central idea consists in averaging the Schr\"odinger operator by the flow of $K_0$ where $K_0=\sqrt{\Delta_g}+Q$ with $Q$ is a pseudo of order $-1$ is chosen (following \cite{Colin}) such that the spectrum of $K_0$  is included in $ \mathbb{N} + \lambda$ for some constant $\lambda\in \mathbb{R}^+$. This crucial property makes the  $K_0$ flow  periodic and motivates us  to use it to average the original operator: if $A$ is a pseudo-differential operator of order $\delta$ then its average with respect to the flow of $K_0$ (see \eqref{averA}) is independent of $t$. In addition the homological equation \eqref{omo12} has a solution $S$ of order $\delta$ and thus $M\mapsto {\rm ad}_{\ii S}M$ maps a pseudo of order $m$ to a pseudo of order $m+\delta-1<m$ (see Lemma \ref{lemmaS}) .
This idea was already used in a pioneering work of Weinstein \cite{Wein}. In \cite{BGMR2} such a procedure was iterated to obtain an equivalent equation like \eqref{nls2} but with $Z$ still depending on time (typically $Z=\langle A\rangle$,  see \eqref{nuovoL+2}). In this paper we alternate the averaging procedure with a time elimination procedure based on the use of the operator \eqref{defT} which solves the homological equation \eqref{omeoequaTT} and thus  the Lie transform $\Phi_T=e^{\ii T}$ will kill the dependence on time in $Z=\langle A\rangle$ (see Lemma \ref{lemmaT}). This time elimination procedure requires a non resonance hypothesis   on the frequency vector $\om$ (see \eqref{zeroMel}).\\
Throughout section \ref{sec3} we work at the pseudo-differential level and the main difficulty is to precisely control the flow generated by pseudo-differential operator of positive order (see   Appendix \ref{AppC} and in particular hypothesis \eqref{lemmaB111}). We notice that all this section holds true upon the hypothesis $\delta<1$.

\medskip

In the KAM step we kill the remainder term $R$ in \eqref{nls2} which still depends on time but  is now a regularizing operator. As in \cite{BBHM} (see also \cite{Mon18} and \cite{BLM18}) we use a reducibility scheme where the regularizing property of the perturbation compensates the bad non resonance estimates satisfied by the eigenvalues of $\Delta_g+\eps Z$ (see \eqref{calO+}). The condition $\delta\leq 1/2$ is used to ensure that condition \eqref{calO+} is prserved during the KAM iteration as long as a small part of the parameters $\om$ are excised  (see Lemma \eqref{lem:misuro} where $\ka=1-2\delta$). This constraint in the KAM procedure was not necessary in \cite{BLM18} (they obtain the reducibility for perturbation of order $1-\mathtt e$ for any $\mathtt e>0$ when the transport operator is of order $1$) essentially because the unperturbed system is integrable. In the context of the transport equation,  the integrability allows Bambusi-Langela-Montalto to prove that the perturbed eigenvalues have the form, $\lambda_j=\lambda_j^{(0)}+ z(j)+ {\rm remainder}$, where $z$ is the symbol of  $Z$  (see formula 
 (4.43) in \cite{BLM18}).  In our case we  just know that $Z$ commutes with $\Delta_g$ and thus we can just prove  that the spectrum of $\Delta_g +V$ preserves the cluster structure inherited from  $\Delta_g$ on a Zoll manifold. That means that, once written in the laplacian diagonalization basis, the matrix of Z is block-diagonal. By the way throughout section \ref{sec4} we work at the matrix level.  \\
 As usual the homological equation \eqref{omoeq} is solved blockwise and it is well known that the increasing size of the blocks may generate loss of regularity. In \cite{EK} Eliasson-Kuksin used geometrical arguments (related to a Bourgain's Lemma, see Lemma 8.1 in \cite{Bou}) to control the size of the blocks, in \cite{GP} or \cite{FG1} authors used  a different argument introduced by Delort-Szeftel in \cite{DS} (see Lemma 4.3 in \cite{GP}). In this paper, as a consequence of the regularization step, we can solve  the homological equation with loos of regularity and thus this step is simplified.\\
   On the other hand the KAM procedure of  \cite{BBHM}  requires a tame property to deal with product of matrices.
 This  motivates the definition of the space $\M_s$ of matrices with $s$-decay norm   (see Definition \ref{decayNorm}) which was first introduced in \cite{BCP} (see also \cite{BP11}). The tame property for  the $s$-decay norm is stated in Lemma \ref{lem:tame}. It is crucial to obtain \eqref{nuovoRem1} and \eqref{nuovoRem1bis} which express the control of the new remainder $R_+$ after one KAM step in two different norms, a low $s$-decay norm and a high $s+\mathtt b$-decay norm. The parameter $N$ measures the troncature in the Fourier variable associate to the angle $\vphi=\om t$ and in the off-diagonal distance in the matrix (see \eqref{cutOffLem}).  When iterating the procedure, this special form of estimates \eqref{nuovoRem1}-\eqref{nuovoRem1bis} allows to obtain a convergent scheme for the sequence of remainders $R_k$ when choosing conveniently the sequence of troncature  parameter $N_k$.

\medskip

Section \ref{sec3} and section \ref{sec4} are independent and in fact are at different levels: while all section 3 takes place in the context of pseudo-differential operators, all section 4 takes place at matrix level. In section \ref{sec5} we merge the two sections and for that we need the Lemma \ref{matrixest} which makes the link between $\rho$-smoothing operators and $\beta$- regularizing matrices.

\medskip

\noindent
{\bf Notation}. 
We shall 
use the notation $A\lesssim B$ to denote 
$A\le C B$ where $C$ is a positive constant 
depending on  parameters fixed once for all: $d$, $n$, $\delta$. 
We shall  use the notation $A\leq_s B$ to denote $A\le C(s) B$ 
where $C(s)>0$ is a constant depending also on $s$.


\section{Functional setting}

In this section we introduce the space of functions, sequences, linear operators and pseudo differential operators we shall use along the paper. 
\subsection{Spectral decomposition}
Following
Theorem 1 of Colin de Verdi\`ere \cite{Colin},
we introduce $Q$ the  pseudo-differential operator  of order $-1$, commuting 
with $\Delta_g$  such that,
setting 
\begin{equation}\label{K0H0}
K_0:=\sqrt{\Delta_g}+Q\,, 
\end{equation} 
we have
$ {\rm spec}(K_0) \subset \mathbb{N} + \lambda$ for some constant $\lambda\in \mathbb{R}^+$. We notice that our original Schr\"odinger operator $H(t):=\Delta_g+\e V(\om t)$ reads
\begin{equation}\label {H}H(t)=\Delta_g +\eps V(\om t)= K_0^2+Q_0+\eps  V(\om t)\end{equation}
where $Q_0=-2Q\sqrt{\Delta_g}-Q^2$ is a pseudo differential operators of order 0.

Let us denote by $\lambda_{k}$ the eigenvalue of $K_0$
and by 
$E_{k}$ be the eigenspace
associated to $\lambda_{k}$. 
We have
\begin{equation}\label{dimension}
\begin{aligned}
&\lambda_{k}\sim k \\
&{\rm dim}E_{k}:=d_{k}\leq k^{n-1}\,.
\end{aligned}
\end{equation}
We denote by 
\begin{equation}\label{ortobase}
\Phi_{[k]}(x):=\{\Phi_{k,m}(x)\,, m=1,\ldots,d_k\}
\end{equation}
an orthonormal  
basis of $E_k$.
By formula \eqref{K0H0} we also deduce that
\begin{equation}\label{lapla}
\Delta_g:=K^2_0+Q_0\,,
\end{equation}
where 
$Q_0$
is a pseudo differential operator commuting both with the Laplacian $\Delta_g$
and $K_0$.
For this reason $K_0$ and $\Delta_g$ diagonalizes simultaneously, hence
\begin{equation}\label{dimensiondelta}
\Delta_g \Phi_{k,j}=\Lambda_{k,j}\Phi_{k,j}\,,\qquad k\in \mathbb{N}\,,\;\;\; j=1,\ldots, d_{k}\,,
\end{equation}
with
\[
\Lambda_{k,j}=\lambda_{k}^{2}+\eta_{k,j}\,,\qquad |\eta_{k,j}|\lesssim 1\,.
\]
In particular there exists $c_0>0$ such that 
\begin{equation}\label{spec}
\Lambda_{k,j}\geq  c_0 k^{2}\,,\qquad |\Lambda_{k,j}-\Lambda_{k',j'}|\geq c_0( k+k')\,,
\qquad \forall \; k\neq k'\,,
\end{equation}
and for any $j=1,\ldots, d_{k}$, $j'=1,\ldots,d_{k'}$.

\subsection{Space of functions and sequences}

Using the spectral decomposition of the space 
$L^{2}(M;\mathbb{C})=\oplus_{k\in\N}E_k$, 
any function $u\in L^{2}({M};\mathbb{C})$
can be written as
\begin{equation}\label{FouExp}
\begin{aligned}
u(x)&=\sum_{k\in \mathbb{N}}\sum_{m=1}^{d_k}z_{k,m}\Phi_{k,m}(x)=\sum_{k\in \mathbb{N}}z_{[k]}\cdot\Phi_{[k]}(x)\,,\qquad \\
& z_{[k]}=(z_{k,1},\cdots,z_{k,d_k})\in \mathbb{C}^{d_{k}}\,,
\end{aligned}
\end{equation}
where $''\cdot ''$ denotes the usual scalar product in $\mathbb{R}^{d_{k}}$.
We denote 
by $\Pi_{E_k}$ the $L^{2}$-projector 
on the eigenspace $E_k$, i.e., for any  $k\in\mathbb{N}$,
\begin{equation}\label{proiettore}
(\Pi_{E_{k}}u)(x)=z_{[k]}\cdot \Phi_{[k]}(x)
\qquad \Rightarrow \qquad 
(\sqrt{-\Delta } +Q)\Pi_{E_k}u=\lambda_{k}\Pi_{E_k}u\,.
\end{equation}
For $s\geq0$, we define the (Sobolev) scale of Hilbert sequence
spaces
\begin{equation}\label{spseq}
\begin{aligned}
h_{s}:=\big\{ z=&\{z_{[k]}\}_{k\in \mathbb{N}}\,, z_{[k]}\in 
\mathbb{C}^{d_{k}}\,:\, \\
&\qquad \qquad  \|z\|^{2}_{h^s}:=
\sum_{k\in \mathbb{N}} \langle k\rangle^{2s}\|z_{[k]}\|^{2}  <+\infty \big\}\,,
\end{aligned}
\end{equation}
where $\langle k\rangle:= \sqrt{1+|k|^2}$ and $\|\cdot\|$ 
denotes
the $L^{2}(\mathbb{C}^{d_{k}})$-norm.
By a slight abuse of notation we define the operator $\Pi_{E_k}$ on sequences 
as
$\Pi_{E_k}z=z_{[k]}$ for any $  z\in h^{s}$
and $k\in \mathbb{N}$. \\
We notice that the weight $\langle k\rangle$ 
we use in the norm in 
\eqref{spseq} is related 
to the eigenvalues of $K_0$, indeed
\begin{equation}\label{equiweight}
c|k|\leq  \lambda_{k}\leq C|k|
\end{equation}
for some suitable constants $0<c\leq C$. \\
As a consequence the space
$$H^s=H^s({\mathtt{M}^n},\C):=\{u(x)= \sum_{k\in \mathbb{N}}z_{[k]}\cdot\Phi_{[k]}(x)\mid z\in h^s\}$$
is the usual Sobolev space $H^s={\rm dom}((K_0)^{s})={\rm dom}(\sqrt{1+\Delta_g})^{s}$  
and
$\|u\|_{H^{s}}:=\|z\|_{h^{s}}$ is equivalent to the standard Sobolev norm $\|u\|_{H^{s}}\sim \|K_0^su\|_{L^2(\mathtt{M}^n)}$.

\medskip

In the paper we shall also deal with quasi periodic in time
functions    $\R\times M\ni(t,x)\mapsto u(\om t,x)$ where $\om\in\R^d$ is a frequency vector and $u$ is periodic in its first variable. To this end we introduce   the space
$H^{r}(\mathbb{T}^{d};H^{s}(\mathtt{M}^{s};\mathbb{C}))$
defined as the set of functions
$ 
u: \mathbb{T}^{d}\ni\vphi\mapsto H^{s}(\mathtt{M}^n;\mathbb{C}) $  
which are Sobolev in $\vphi\in \mathbb{T}^{d}$ with values in 
$H^{s}(\mathtt{M}^n;\mathbb{C})$.

Functions in $H^{r}(\mathbb{T}^{d};H^{s}(\mathtt{M}^{s};\mathbb{C}))$
can be expanded, using the standard Fourier theory, as
\begin{equation}\label{decompogo}
u(\vphi,x)=\sum_{\ell\in \mathbb{Z}^{d}, k\in \mathbb{N}}z_{[k]}(l)\cdot
 \Phi_{[k]}(x) e^{\ii l \cdot\vphi}\,, \qquad z_{[k]}(l)\in \mathbb{C}^{d_k}
\end{equation}
where $e^{\ii l\cdot\vphi} \Phi_{k,m}(x)$, 
$l\in \mathbb{Z}^{d}$, $k\in \mathbb{N}$, $m=1,\ldots,d_k$
is an orthogonal basis of 
$L^{2}(\mathbb{T}^{d}\times \mathtt{M}^{n};\mathbb{C})$.
We define space of sequence 
(recall \eqref{spseq})
\begin{equation}\label{timeseq}
\begin{aligned}
h_{s,r}&:=\big\{
z=\{z_{[k]}(l)\}_{l\in \mathbb{Z}^{d},k\in\mathbb{N}}\,,\, z_{[k]}\in 
\mathbb{C}^{d_k}:\,\\
&\qquad \qquad \quad
\| z\|^{2}_{h_{s,r}}:=\sum_{l\in \mathbb{Z}^{d}}\langle l\rangle^{2r}
\|z(l)\|_{H^s}^{2}<+\infty\big\}\,.
\end{aligned}
\end{equation}
Along the paper we shall also consider the space, for $p\in \mathbb{N}$ with $p> \frac{d+n}{2}$,
\begin{equation}\label{spaziocomp}
\ell_{p}:=\bigcap_{\substack{
r> d/2, s> n/2
\\ s+r=p}} h_{s,r}\,.
\end{equation}
We endow the space $\ell_{p}$ with the  norm
\begin{equation}\label{elleNorma}
\|z\|_{\ell_p}^{2}:=\sum_{l\in\mathbb{Z}^{d},k\in \mathbb{N}}\langle l,k\rangle^{2p}\|z_{[k]}(l)\|^{2}\,.
\end{equation}


\vspace{0.9em}
\noindent
{\bf Lipschitz norm.} Consider a compact subset $\mathcal{O}$ of
$\mathbb{R}^{d}$, $d\geq1$. 
For functions $f : \mathcal{O}\to E$, with $(E,\|\cdot\|_{E})$
some Banach space, we define 
the sup norm and the lipschitz semi-norm
as
\begin{equation}\label{suplip}
\begin{aligned}
\|f\|_{E}^{sup}&:=\|f\|_{E}^{sup,\mathcal{O}}:=\sup_{\omega\in \mathcal{O}}
\|f(\omega)\|_{E}\,,\\
\|f\|_{E}^{lip}&:=\|f\|_{E}^{lip,\mathcal{O}}:=\sup_{\substack{\omega_1,\omega_2\in \mathcal{O}\\ \omega_1\neq\omega_2}}\frac{\|f(\omega_1)-f(\omega_2)\|_{E}}{|\omega_1-\omega_2|}\,.
\end{aligned}
\end{equation}
For any $\gamma>0$ we introduce the weighted Lipschitz norms
\begin{equation}\label{Lip-norm}
\|f\|_{E}^{\gamma,\mathcal{O}}:=\|f\|_{E}^{sup,\mathcal{O}}
+\gamma \|f\|_{E}^{lip,\mathcal{O}}\,.
\end{equation}
%
%
We finally define the space of sequences 
\begin{equation}\label{spaseqLip}
h_{s,r}^{\gamma,\mathcal{O}}:=\big\{
\mathcal{O}\ni\omega\mapsto z(\omega)\in h_{s,r}\,:\, 
\|z\|_{h_{s,r}}^{\gamma,\mathcal{O}}<+\infty
\big\}\,,
\end{equation}
and consequently the space (recall \eqref{spaziocomp})
\begin{equation}\label{spaziocompLip}
\ell_{p}^{\gamma,\mathcal{O}}:=\bigcap_{s+r=p}h_{s,r}^{\gamma,\mathcal{O}}\,,
\end{equation}
endowed with the norm
\begin{equation}\label{elleNormaLip}
\|z\|_{\ell_p}^{\gamma,\mathcal{O}}
:=\|z\|_{\ell_p}^{sup,\mathcal{O}}
+\gamma \|z\|_{\ell_{p}}^{lip,\mathcal{O}}\,.
\end{equation}

\subsection{Pseudo-differential operators}
In this paper we consider operators which are \emph{pseudo-differential}.
Here we recall some 
 fundamental properties of 
 operators in $\mathcal{A}_{m}$ 
which are collected in  \cite{BGMR2}. 
First $\mathcal{A}_{m}$ is  a Fr\'echet space for a family of  
filtering semi-norms $\{ \mathcal{N}_{m,p}\}_{p\geq 1}$  
such that the embedding 
$\A_m\hookrightarrow  \bigcap_{s \in \R} \mathcal L(H^s,H^{s-m})$ is continuous. 
We can also chose the semi-norms in an increasing way, 
i.e.  $ \mathcal{N}_{m,p} (A)\leq  \mathcal{N}_{m,p+1} (A)$ 
for $p\geq 1$ and $A\in\mathcal{A}_{m}$. 
To state the other properties we need to introduce the following definition.
\begin{definition}\label{Nsmooth}
Let $S\in \mathcal{L}(H)$.
We say that $S$ is $\rho$-smoothing, and we will write $S\in\mR_\rho$, if $S$ can be extended to 
an operator in
$\mathcal{L}(H^{s}, H^{s+\rho})$
for any $s\in \mathbb{R}$.
When this is true for every 
$\rho\geq 0$,
we say that $S$ is a smoothing operator.
\end{definition}

\noindent
Then  we have the following properties concerning the class  $\mathcal{A}_{m}$ equipped with the semi-norms $\{ \mathcal{N}_{m,p}\}_{p\geq 1}$:

\begin{itemize}

\item[$(i)$] let $A\in \mathcal{A}_{m}$, for any $s\in\mathbb{R} $
 there exist constants $C=C(m, s)>0$, $p=p(m,s)\geq1$ which is an increasing function\footnote{This fact is quite evident in the case of pseudo-differential operators on $\R^n$ and thus extends to pseudo-differential operators on $\mathtt M^n$ by passing to local charts.} of $s$ such that

\begin{equation}\label{action}
 \|A\|_{\mathcal{L}(H^s, H^{s-m})} \le  C \mathcal{N}_{m,p} (A)\,.
 \end{equation}

 \item[$(ii)$] Let $ A \in \mathcal{A}_m, \, B \in \mathcal{A}_n$ then 
 $ AB \in \mathcal{A}_{m+n}$. Furthermore  for any $\rho\geq0$ there exists $S$ a  $\rho$-smoothing operator such that for any $p\geq1$ for any $s\in\R$ there are constants
 $C=C(m,n, p,s,\rho)>0$, $q=q(m,n,p,s,\rho)\geq p$ such that

\begin{align}\label{compoPseudo}
 \mathcal{N}_{m+n,p} ( AB-S)
 \le C \mathcal{N}_{m,q} (A) \mathcal{N}_{n,q} (B)\,,\\
 \label{compoPseudo2}\|S\|_{\mathcal{L}(H^{s}, H^{s+\rho})} \le C \mathcal{N}_{m,q} (A) \mathcal{N}_{n,q} (B)\,.
\end{align}

 \item[$(iii)$] Let $ A \in \mathcal{A}_m, \, B \in \mathcal{A}_n$ then 
 $ [A, B] \in \mathcal{A}_{m+n-1}$. Furthermore  for any $\rho\geq0$ there exists $S$ a  $\rho$-smoothing operator such that for any $p\geq1$ for any $s\in\R$ there are constants
 $C=C(m,n, p,s,\rho)>0$, $q=q(m,n,p,s,\rho)\geq p$ such that

\begin{align}\label{commuPseudo}
\mathcal{N}_{m+n-1,p} ( [A,B]-S) \le C \mathcal{N}_{m,q} ( A)\mathcal{N}_{n,q} ( B)\,,\\
 \label{commuPseudo2}\|S\|_{\mathcal{L}(H^{s}, H^{s+\rho})} \le C \mathcal{N}_{m,q} (A) \mathcal{N}_{n,q} (B)\,.
\end{align}

\item[$(iv)$] The map $\t \to A(\t):= e^{-itK_0} A e^{i tK_0} \in C^0_b (\mathcal{R}, \mathcal{A}_m).$ Furthermore  for any $\rho\geq0$ there exists $S$ a  $\rho$-smoothing operator such that for any $p\geq1$ for any $s\in\R$ there are constants
 $C=C(m,n, p,s,\rho)>0$, $q=q(m,n,p,s,\rho)\geq p$ such that

\begin{align}\label{ego}
\mathcal{N}_{m+n-1,p} (  e^{-itK_0} A e^{i tK_0}-S) \le C \mathcal{N}_{m,q} ( A)\,,\\
 \label{ego2}\|S\|_{\mathcal{L}(H^{s}, H^{s+\rho})} \le C \mathcal{N}_{m,q} (A) \,.
\end{align}

\end{itemize}
\begin{remark}\label{rempseudo}In (ii), (iii) and (iv) the smoothing correction doesn't play an important role since it can be chosen as regularizing as one want. In the KAM scheme the level of regularization will be fix once for all. Thus, 
by a slight abuse of notation, we will offen omit in the following the smoothing correction and will just write
\begin{align}\label{compoPseudo3}
 \mathcal{N}_{m+n,p} ( AB)
 \le C \mathcal{N}_{m,q} (A) \mathcal{N}_{n,q} (B)\,,\\
\label{commuPseudo3}\mathcal{N}_{m+n-1,p} ( [A,B]) 
\le C \mathcal{N}_{m,q} ( A)\mathcal{N}_{n,q} ( B)\,.
\end{align}
\end{remark}

We shall also consider $H^r$-mappings
 \begin{equation}\label{pseudofam}
 \mathbb{T}^{d}\ni\vphi\mapsto A(\vphi)
 \end{equation}
 with $A(\vphi)$
 a 
 symmetric pseudo-differential operators of order $m$ in $\mathcal{A}_{m}$. We can then decompose $A$ in Fourier writing
 \begin{equation}\label{pseudofam1}
 A(\vphi)=\sum_{l\in\mathbb{Z}^{d}}A(l)e^{\ii l\cdot\vphi}
 \end{equation}
 with $A(l)$
 a  pseudo-differential operators of order $m$ in $\mathcal{A}_{m}$. 
We give the following definition.
 
 \begin{definition}\label{def:pseudotempo}
 Let $m\in \mathbb{R}$, $r>d/2$. We denote by 
 $\mathcal{A}_{m,s}$ the Fr\'echet space  of mapping $\mathbb{T}^{d}\ni\vphi\mapsto A=A(\vphi)\in\A_m$ that have $H^r$ on $\T^d$. We endow $\mathcal{A}_{m,r}$ with the family of semi-norms
 \begin{equation} \label{psdo}
 \big(\Nc_{m,r,p}(A)\big)^{2}:=\sum_{\ell \in\Z^d} \langle \ell\rangle^{2r} \mathcal{N}^2_{m,p} (A(l)) ,\quad p\geq1\,.
\end{equation}
  Consider a Lipschitz family 
$\mathcal{O}\ni\omega\mapsto A(\om)\in \mathcal{A}_{m,r}$
where $\mathcal{O}$ is a compact subset of $\mathbb{R}^{d}$, $d\geq 1$.
For $\gamma>0$ we define the Lipschitz semi-norms (recall \eqref{suplip})
as
\begin{equation}\label{decayLipPseudo}
\Nc_{m,r,p}^{\gamma,\mathcal{O}}(A)
:=\Nc_{m,r,p}^{sup,\mathcal{O}}(A)
+\gamma
\Nc_{m,r,p}^{lip,\mathcal{O}}(A)
\end{equation}
We denote by $\mathcal{A}^{\gamma,\mathcal{O}}_{m,r}$ the Fr\'echet space of 
families of pseudo differential operators $A(\omega)\in \mathcal{A}_{m,r}$ endowed with 
with   the family of semi-norms $\{\Nc_{m,r,p}^{\gamma,\mathcal{O}}\}_{p\geq1}$.
 \end{definition}
 
Similarly  we define the corresponding class of  $\rho$-smoothing operators $R(\om,\vphi)$, $H^r$ in $\vphi$ and Lischitz in $\om$.
\begin{definition}\label{Rrho} 
Let $\rho\in\R$ and $r>d/2$.
We denote by $\mathcal{R}_{\rho,r}$ the Fr\'echet space
 of $\rho$-smoothing $H^r$-mapping 
 $\mathbb{T}^{d}\ni\vphi\mapsto R(\vphi)\in\L(H^s,H^{s+\rho})$ 
 for all $s\in\R$  endowed 
 with the family of semi-norms 
\begin{equation}\label{flusso50}
|R|^2_{\rho,r,s} 
:=\sum_{\ell \in\Z^d} \langle \ell\rangle^{2r}\|R(l)\|^2_{\L(H^s,H^{s+\rho})} 
\quad s\in\R\,.
\end{equation}
Consider a family 
$\mathcal{O}\ni\omega\mapsto R(\omega)\in \mathcal{R}_{\rho,r}$
where $\mathcal{O}$ is a compact subset of $\mathbb{R}^{d}$, $d\geq 1$.
For $\gamma>0$ we denote by $\mathcal{R}^{\gamma,\mathcal{O}}_{\rho,r}$ 
the Fr\'echet space of 
families of pseudo differential operators $R(\omega)\in \mathcal{R}_{\rho,r}$ endowed with 
with   the family of semi-norms $\{\Nc_{\rho,r,p}^{\gamma,\mathcal{O}}\}_{p\in\N}$ defined by (recall \eqref{suplip})
\begin{equation}
|R|_{\rho,r,p}^{\gamma,\mathcal{O}}
:=|R|_{\rho,r,p}^{sup,\mathcal{O}}
+\gamma
|R|_{\rho,r,p}^{lip,\mathcal{O}}\,.
\end{equation}
\end{definition}
We notice that by \eqref{action} we have $\A_{m,r}\subset \mathcal{R}_{-m,r}$.

\begin{lemma}
Let $r>d/2$, $m,\rho\in \mathbb{R}$ and 
consider $R\in \mathcal{R}_{\rho,r}^{\gamma,\mathcal{O}}$ and 
$A\in \mathcal{A}^{\gamma,\mathcal{O}}_{m,r}$.
Then, for any $s\in\R$, there are $C=C(s,r)>0$, $p(s,m)>0$ such that
\begin{align}
\|A h\|_{h_{s-m,r}^{\gamma,\mathcal{O}}}&\leq
C \Nc_{m,r,p}^{\gamma,\mathcal{O}}(A)\|h\|_{h_{s,r}^{\gamma,\mathcal{O}}}\,,\label{ranma1}\\
\|R h\|_{h_{s+\rho,r}^{\gamma,\mathcal{O}}}&\leq
C |R|_{\rho,r,s}^{\gamma,\mathcal{O}}\|h\|_{h_{s,r}^{\gamma,\mathcal{O}}}\,,\label{ranma2}
\end{align}
for any $h\in h_{s,r}^{\gamma,\mathcal{O}}$.
\end{lemma}

\begin{proof}
We start by  proving the \eqref{ranma2} for the norm $\|\cdot\|_{h_{s+\rho,r}}$. 
Recalling \eqref{timeseq} we have
\begin{equation*}
\begin{aligned}
\|Rh&\|_{h_{s+\rho,r}}\leq \sum_{l\in \mathbb{Z}^{d}}\langle l\rangle^{2r} \Big(
\sum_{l'\in\mathbb{Z}^{d}}\|R(l-l') h(l')\|_{H^{s+\rho}}
\Big)^{2}\\
&\leq
 \sum_{l\in \mathbb{Z}^{d}}\langle l\rangle^{2r} \Big(
\sum_{l'\in\mathbb{Z}^{d}}\|R(l-l')\|_{\mathcal{L}(H^{s};H^{s+\rho})} \|h(l')\|_{H^{s}}
\Big)^{2}\\
&\leq
 \sum_{l\in \mathbb{Z}^{d}} \Big(
\sum_{|l'|>\frac{1}{2}|l|}\langle l-l'\rangle^{r}\|R(l-l')\|_{\mathcal{L}(H^{s};H^{s+\rho})} 
\langle l'\rangle^{r}
\|h(l')\|_{H^{s}}\frac{\langle l\rangle^{r}}{\langle l'\rangle^{r}}
\Big)^{2}\\
&+
 \sum_{l\in \mathbb{Z}^{d}} \Big(
\sum_{|l'|\leq\frac{1}{2}|l|}
\langle l-l'\rangle^{r}\|R(l-l')\|_{\mathcal{L}(H^{s};H^{s+\rho})} 
\langle l'\rangle^{r}
\|h(l')\|_{H^{s}}\frac{\langle l\rangle^{r}}{\langle l-l'\rangle^{r}}
\Big)^{2}\,.
\end{aligned}
\end{equation*}
Hence, by using the 
Cauchy-Schwartz inequality, we get
\begin{equation*}
\begin{aligned}
\|Rh\|_{h_{s+\rho,r}}&\leq C\sum_{l,l'\in \mathbb{Z}^{d}}\langle l-l'\rangle^{2r}
\|R(l-l')\|^{2}_{\mathcal{L}(H^{s};H^{s+\rho})} \langle l'\rangle^{2r}\|h(l')\|_{H^{s}}^2\\
&\leq C\|h\|_{h_{s,r}}^{2}|R|^2_{\rho,r,s}\,,
\end{aligned}
\end{equation*}
which implies the \eqref{ranma2} for the norm $\|\cdot\|_{h_{s+\rho,r}}$.
The Lipschitz bound on the norm  $\|\cdot\|_{h_{s+\rho,r}}^{\gamma,\mathcal{O}}$ 
and the \eqref{ranma1} follows similarly. 
\end{proof}

  In the following Lemma we state some properties and estimates\footnote{Estimates \eqref{algSeminorm}, \eqref{algSeminorm2} and \eqref{algSeminorm3} 
  are written taking into account Remark \ref{rempseudo}. 
  For instance \eqref{algSeminorm2} should be interpreted as:
  for any $\rho\geq0$ there exists $S$ a  
  $\rho$-smoothing operator such that for any $p\geq1$ for any $s\in\R$ there are constants
 $C=C(m,n, p,r,\rho)>0$, $q=q(m,n,p,r,\rho)\geq 1$ such that
\[
\mathcal{N}_{m+n-1,r,p}^{\gamma,\mathcal{O}} ( [A,B]-S) \le C \mathcal{N}_{m,r,q}^{\gamma,\mathcal{O}} ( A)\mathcal{N}_{n,r,q}^{\gamma,\mathcal{O}} ( B)\,,\\
 |S|_{\rho,s,p}^{\gamma,\mathcal{O}} 
 \le C \mathcal{N}_{m,r,q}^{\gamma,\mathcal{O}} (A) 
 \mathcal{N}_{n,r,q}^{\gamma,\mathcal{O}} (B)\,.
 \]
 } that will be proved in Appendix \ref{appendixAA}.
 \begin{lemma}\label{propertiespdo}
 Let $A,\, B$ are pseudo-differential operators in $\mathcal{A}^{\gamma,\mathcal{O}}_{  m,r}$ 
 and $\mathcal{A}^{\gamma,\mathcal{O}}_{ n,r}.$ For any $p\geq1$ there exist constants $C= C(r,m,n,p)$ and $ q= q(r, m,n,p)$ which is increasing in $p$ such that

\noindent
$(i)$ $AB, \,BA \in \mathcal{A}^{\gamma,\mathcal{O}}_{  m+n,r}$ and 
  \begin{equation}\label{algSeminorm}
  \Nc^{\gamma,\mathcal{O}}_{ m+n,r,p}(AB) ,\,  \Nc^{\gamma,\mathcal{O}}_{m+n,r,p} (BA)
  \le C \Nc^{\gamma,\mathcal{O}}_{ m,r,q}(A)\Nc^{\gamma,\mathcal{O}}_{ n,r,q}(B)\,.
  \end{equation}
  
\noindent
$(ii)$ $[A,B] \in \mathcal{A}^{\gamma,\mathcal{O}}_{ m+n-1,r}$  and 
    \begin{equation}\label{algSeminorm2}
   \Nc^{\gamma,\mathcal{O}}_{ m+n-1,r,p}([A,B]) \le C \Nc^{\gamma,\mathcal{O}}_{m,r,q}(A)
   \Nc^{\gamma,\mathcal{O}}_{n,r,q}(B)\,.
\end{equation}

\noindent
$(iii)$
  Let $\omega \in \mathbb{R}^d$, then $\omega\cdot \partial_{\vphi}A\in \A_{m,r-1}$ and
  \begin{equation}\label{estOmega}
  \Nc^{\gamma,\mathcal{O}}_{ m,r-1,p} (\omega\cdot \partial_{\vphi}A)
  \le C \Nc^{\gamma,\mathcal{O}}_{ m,r,p} (A).
  \end{equation}
  If furthermore  $\om$ satisfies, for some $\alpha>d-1$,
  \begin{equation}\label{diodio}
  | \omega \cdot l|> \frac{\gamma}{|l|^{\alpha}} \,,\qquad \forall \; l\in  \mathbb{Z}^d\setminus\{0\}\,,
  \end{equation}
  and $r-2\alpha-1>d/2$ 
  then $(\omega\cdot \partial_{\vphi})^{-1}A\in \A_{m,r-(2\alpha+1)}$ and
  \begin{equation}\label{estOmega-1}
    \Nc^{\gamma,\mathcal{O}}_{ m,r-(2\alpha+1),p} ((\omega\cdot \partial_{\vphi})^{-1}A)
  \le C\gamma^{-1}\Nc^{\gamma,\mathcal{O}}_{ m,r,p} (A).
  \end{equation}

\noindent
$(iv)$
  For any $t \in [0, 2 \pi]$ we have 
  $ e^{-itK_0} A e^{i tK_0} \in \mathcal{A}^{\gamma,\mathcal{O}}_{  m,r}$  
  and 
   \begin{equation}\label{algSeminorm3}
  \Nc^{\gamma,\mathcal{O}}_{ m,r,p}(e^{-itK_0} A e^{i tK_0}) \le C \Nc^{\gamma,\mathcal{O}}_{ m,r,q}(A)\,.
\end{equation}
  \end{lemma}

\subsection{Conjugations rules}
Let $\omega\cdot\pa_{\vphi}$ be the diagonal operator acting on sequences $z\in \ell_{s,r}$ (see \eqref{timeseq}) defined by
\begin{equation}\label{omegaphi}
\omega\cdot\pa_{\vphi}z:={\rm diag}_{l\in \mathbb{Z}^{d}, k\in \mathbb{N}}(\ii \omega\cdot l)z=
(\ii \omega\cdot l z_{[k]}(l))_{l\in \mathbb{Z}^{d}, k\in \mathbb{N}}.
\end{equation}
 Consider an operator of the form
 \begin{equation}\label{operTotale}
 L:=L(\vphi,\omega):=\omega\cdot\pa_{\vphi}+\ii M(\vphi)\,,
 \end{equation}
 where $M(\vphi)$ is some map
 $\mathbb{T}^{d}\ni\vphi\mapsto M=M(\vphi)\in\mathcal{L}(H^{s};H^{s+m})$, 
 for some $m\in \mathbb{R}$.
We shall study how the operator $L$ in \eqref{operTotale} conjugates under the map $\Phi_{S}$
defined as
\begin{equation}\label{operTotale2}
\Phi_{S}:=(\Phi_{S}^{\tau})_{|\tau=1}\,, \qquad \Phi_{S}^{\tau}:=e^{\ii \tau S}
=\sum_{p=0}^{\infty}\frac{1}{p!}(\ii S)^{p}
\end{equation}
where
$S(\vphi)$ is some map
 $\mathbb{T}^{d}\ni\vphi\mapsto S=S(\vphi)\in\mathcal{L}(H^{s};H^{s+m'})$, 
 for some $m'\in \mathbb{R}$.
For the well-posedness of a map of the form \eqref{operTotale2}
we refer to Lemma \ref{lemmaB1} in Appendix \ref{AppC}.

By using the Lie series expansions we have
\begin{equation}\label{operTotale3}
\begin{aligned}
L^{+}=L^{+}(\vphi)&:=\Phi_{S}\circ L\circ \Phi_{S}^{-1}=
\omega\cdot\pa_{\vphi}+\ii M^{+}(\vphi)
\end{aligned}
\end{equation}
where $M^{+}(\vphi)=M_1^{+}(\vphi)+M_{2}^{+}(\vphi)$ with, for any $q\in \mathbb{N}$,
\begin{equation}\label{operTotale4}
\begin{aligned}
\ii M_1^{+}(\vphi):=\Phi_{S}\circ \ii M\circ \Phi_{S}^{-1}&=\ii M+
\sum_{p=1}^{q}\frac{1}{p!}{\rm ad}_{\ii S}^{p}(\ii M)
+
\frac{1}{q!}\int_{0}^{1}(1-\tau)^{q}\Phi_{S}^{\tau}{\rm ad}_{\ii S}^{q+1}(\ii M)\Phi_{S}^{-\tau}d\tau\,,
\end{aligned}
\end{equation}
and 
\begin{equation}\label{operTotale5}
\begin{aligned}
\ii M_2^{+}(\vphi)&:=\Phi_{S}\circ \omega\cdot\pa_{\vphi}\circ \Phi_{S}^{-1}-\omega\cdot\pa_{\vphi}\\
&=
-\ii \omega\cdot\pa_{\vphi}S-
\sum_{p=2}^{q}\frac{1}{p!}{\rm ad}_{\ii S}^{p-1}(\ii  \omega\cdot\pa_{\vphi}S)
+
\frac{1}{q!}\int_{0}^{1}(1-\tau)^{q}\Phi_{S}^{\tau}{\rm ad}_{\ii S}^{q}(\ii  \omega\cdot\pa_{\vphi}S)\Phi_{S}^{-\tau}d\tau\,,
\end{aligned}
\end{equation}
where we defined ${\rm ad}^{0}_{ S}( M)=M$ and 
\begin{equation}\label{commu}
{\rm ad}^{p}_{ S}(M)=
{\rm ad}^{p-1}_{ S}([ S, M])\,,\quad 
[S, M]= S M-M S\,.
\end{equation}

\begin{remark}{\bf (Hamiltonian structure)} We remark that, if the operator 
$S$ in and $M$ are \emph{Hermitian}, then by Lemma $2.9$ in \cite{FG1},
we have that also the operator $M^{+}$ in \eqref{operTotale3}
is Hermitian.

\end{remark}

\subsection{Linear operators and matrices.}

According to the 
orthogonal splitting
\begin{equation}\label{orthosplitto}
L^{2}(\mathbb{S}^{n};\mathbb{C})=\bigoplus_{k\in \mathbb{N}} E_{k},
\end{equation}
we identify a linear operator 
acting on $L^{2}(M;\mathbb{C})$
with its matrix 
representation $A:=\Big(A_{[k]}^{[k']}\Big)_{k,k'\in\mathbb{N}}$
in $\mathcal{L}(h^{0})$ (recall \eqref{spseq})
with blocks $A_{[k]}^{[k']}\in \mathcal{L}(E_{k'};E_{k})$.
Notice that each block
$A_{[k]}^{[k']}$ is a $d_{k}\times d_{k'}$:
\begin{equation}\label{notazioneBlock}
A_{[k]}^{[k']}:=\Big(A_{k,j}^{k',j'}\Big)_{\substack{j=1,\ldots, d_{k},\\ j'=1,\ldots,d_{k'}}}\,.
\end{equation}
The action of the operator $A$ on functions $u(x)$ as in \eqref{FouExp} 
of the space variable
in $L^{2}(\mathbb{S}^{n};\mathbb{C})$ is given by
\begin{equation}\label{notazioneBlock2}
(Au)(x)=\sum_{k\in \mathbb{N}}(Az)_{[k]}\cdot\Phi_{[k]}(x)\,,\qquad z_{[k]}\in \mathbb{C}^{d_{k}}\,,\qquad 
(Az)_{[k]}=\sum_{j\in \mathbb{N}}A_{[k]}^{[j]}z_{[j]}\,.
\end{equation}

Given $s, s'\in \mathbb{R}$ we denote by $\mathcal{L}(H^{s}, H^{s'})$
the space of linear bounded operators  form $H^{s}$ to $H^{s'}$ endowed 
with the standard operator  norm $\|\cdot\|_{\mathcal{L}(H^{s}, H^{s'})}$.

In this paper we also consider regular
 $\vphi$-dependent families of linear operators 
\begin{equation}\label{timeMat}
\mathbb{T}^{d}\ni\vphi\mapsto A=A(\vphi)=\sum_{l\in\mathbb{Z}^{d}}A(l)e^{\ii l\cdot\vphi}
\end{equation}
where $A(l)$ are linear operators in $\mathcal{L}(H^{s}, H^{s'})$,
for any $ l\in \mathbb{Z}^{d}$\,.
We also regard $A$ as an operator acting on functions 
$u(\vphi,x)$ of space-time 
as
\[
(Au)(\vphi,x)=(A(\vphi)u(\vphi,\cdot))(x)\,.
\]
More precisely, expanding $u$ as in \eqref{decompogo},
we have
\begin{equation}\label{azione}
\begin{aligned}
(Au)(\vphi,x)&=\sum_{l\in \mathbb{Z}^{d},k\in \mathbb{N}}(Az)_{[k]}(l)
e^{\ii l\cdot\vphi}\Phi_{[k]}(x)\,,\\
(Az)_{[k]}(l)&=\sum_{p\in \mathbb{Z}^{d},k'\in N}A_{[k]}^{[k']}(l-p)z_{[k']}(p)\,.
\end{aligned}
\end{equation}
Relation \eqref{notazioneBlock2} shows that, in order to define operators  that conserve the $H^s$ regularity in space we need to assume some decay of $\|A_{[k]}^{[k']}\|^{2}_{\mathcal{L}(L^{2})}$ with respect to $|k-k'|$. That the reason for the following definition first introduced in \cite{BP} for (i) and in \cite{BCP} for (ii)

\begin{definition}{\bf ($s$-decay norm)}\label{decayNorm}

\noindent(i) We define the $s$-decay norm of a matrix $A\in \mathcal{L}(H^{s};H^{s})$ as 
\begin{equation}\label{decayNorm1}
| A|_{s}^{2}:=\sum_{h\in \mathbb{N}}\langle h\rangle^{2s}
\sup_{|k-k'|=h}\|A_{[k]}^{[k']}\|^{2}_{\mathcal{L}(L^{2})}
\end{equation}
where $\|\cdot\|_{\mathcal{L}(L^{2})}$ is the $L^{2}$-operator norm 
in $\mathcal{L}(E_{k'},E_k)$.

\noindent (ii) 
Consider a map 
$\mathbb{T}^{d}\ni\vphi\mapsto A=A(\vphi)\in \mathcal{L}(H^{s};H^{s})$. 
We define its decay norm as
\begin{equation}\label{decayNorm2}
\bral A\brar_{s}^{2}:=\sum_{l\in \mathbb{Z}^{d}, 
h\in \mathbb{N}}
\langle l,  h\rangle^{2s}
\sup_{|k-k'|=h}\|A_{[k]}^{[k']}(l)\|^{2}_{\mathcal{L}(L^{2})} 
\end{equation}
We denote by $\mathcal{M}_{s}$ the space  matrices with finite 
$s$-decay norm $\bral \cdot\brar_{s}$.

\noindent (iii) Consider a  Lipschitz family 
$\mathcal{O}\ni\omega\mapsto A(\omega)\in \mathcal{M}_{s}$
where $\mathcal{O}$ is a compact subset of $\mathbb{R}^{d}$, $d\geq 1$.
For $\gamma>0$ we define the Lipschitz decay norm
as
\begin{equation}\label{decayLip}
\begin{aligned}
\bral A\brar^{\gamma,\mathcal{O}}_{s}&:=\bral A\brar^{sup,\mathcal{O}}_{s}
+\gamma
\bral A\brar^{lip,\mathcal{O}}_{s}\\
&=
\sup_{\omega\in \mathcal{O}}\bral A(\omega)\brar_{s}+
\gamma
\sup_{\substack{\omega_1,\omega_2\in \mathcal{O}\\ 
\omega_1\neq\omega_2}}\frac{\bral A(\omega_1)-A(\omega_2)\brar_{s}}{|\omega_1-\omega_2|}\,.
\end{aligned}
\end{equation}
We denote by $\mathcal{M}^{\gamma,\mathcal{O}}_{s}$ the space of 
families of Lipschitz mapping $\om\mapsto A(\omega)\in \mathcal{M}_{s} $
with finite  $|\cdot|_{s}^{\gamma,\mathcal{O}}$-norm.
\end{definition}
\begin{remark}
The $s$-decay norm \eqref{decayNorm2} link the regularity in space and the regularity in $\phi$ (i.e. in time). In fact for $s$ integer we have
$$\mathcal M_s =\cap_{p+q\leq s}H^p(\T^d, \L(H^q(\mathtt M^n),H^q(\mathtt M^n))).$$

\end{remark}
\begin{remark}
Notice that, if the $s$-decay norm of a matrix $A$ is finite, then
\[
\|A_{[k]}^{[k']}\|_{\mathcal{L}(L^{2})}\leq C(s) \bral A\brar_{s} \langle k-k'\rangle^{-s}\,.
\]
\end{remark}
We have the following fondamental lemma stating in particular  that the $s$-decay norm is tame (see \eqref{decayTame}).  This tame property will be crucial in the KAM procedure.

\begin{lemma}\label{lem:tame}
For any $s>(d+n)/2$ the following holds:

\noindent
$(i)$ 
there is $C=C(s)>0$ such that (recall  \eqref{spaziocomp},\eqref{elleNorma})

\begin{equation}\label{azio}
\|Az\|_{\ell_{s}}\leq C \bral A\brar_{s}\|z\|_{\ell_{s_0}}+C
 \bral A\brar_{s_0}\|z\|_{\ell_{s}}\,,
\end{equation}
for any $h\in \ell_{s}$;

\noindent
$(ii)$ there is $C=C(s)>0$ such that
\begin{equation}\label{decayTame}
\bral AB\brar_{s}\leq  C\bral A\brar_{s}\bral B\brar_{s_0}+C \bral A\brar_{s_0}\bral B\brar_{s}\,;
\end{equation}

\noindent
$(iii)$ 
for $N>0$ we define  (recall \eqref{timeMat})
the matrix $\Pi_{N}A$ as
\begin{equation}\label{cutOff}
(\Pi_{N}A)_{[k]}^{[k']}(l):=\left\{
\begin{aligned} &A_{[k]}^{[k']}(l)\,,\quad l\in \mathbb{Z}^{d}\,, k,k'\in\mathbb{N}\,,\quad
\begin{aligned}
&|l|\leq N\,, \\ &|k-k'|\leq N\,,\end{aligned}
\\
&0\,,\qquad {\rm otherwise}
\end{aligned}
\right.
\end{equation}
One has
\begin{equation}\label{smoothesti}
\bral ({\rm Id}- \Pi_{N})A\brar_{s}\leq C
N^{-\beta} \bral A\brar_{s+\beta}\,,\quad \beta\geq 0\,,
\end{equation}
for some $C=C(s)>0$.

Similar bounds holds also replacing $\|\cdot\|_{\ell_{s}}$, $\bral \cdot\brar_{s}$
with the norms $\|\cdot\|_{s}^{\gamma,\mathcal{O}}$, 
$\bral\cdot\brar_{s}^{\gamma,\mathcal{O}}$ respectively
(see \eqref{elleNormaLip}, \eqref{decayLip}).

\end{lemma}
\begin{proof}
Items $(i)$ and $(ii)$ follow by lemmata $2.6$, $2.7$ in \cite{BCP}.
Item $(iii)$ follows by the definition of the norm in \eqref{decayNorm2}.
\end{proof}

We will also need  a  class of matrices that take into account a notion of regularization.
\begin{definition}
\label{1smooth}
Define the diagonal $\vphi$-independent operator  $\mathcal{D}$, 
acting on 
$z\in \ell_{s}$ (see \eqref{spaziocomp}), as 
\begin{equation}\label{Diag}
\mathcal{D}z:=
{\rm diag}_{l\in \mathbb{Z}^{d}, k\in 
\mathbb{N}}\big(\lambda_{k}\big)z=
\big(\lambda_{k} z_{[k]}(l)\big)_{l\in \mathbb{Z}^{d}, k\in \mathbb{N}}\,.
\end{equation}
For $\beta\in \mathbb{R}$
we define the  norm $\bral \cdot \brar_{\beta,s}$ 
of a matrix $A$  in \eqref{timeMat}
as
\begin{equation}\label{decayNorm3}
\bral A\brar_{\beta, s}:=\bral\mathcal{D}^{\beta}A\brar_{s}+ \bral A\mathcal{D}^{\beta}\brar_{s}\,.
\end{equation}
We denote by $\mathcal{M}_{\beta,s}$ the space of  maps $\mathbb{T}^{d}\ni\vphi\mapsto A=A(\vphi)\in\L(L^2)$   with finite $\bral \cdot\brar_{\beta,s}$-norm.

Consider a family 
$\mathcal{O}\ni\omega\mapsto A(\omega)\in \mathcal{M}_{\beta,s}$
where $\mathcal{O}$ is a compact subset of $\mathbb{R}^{d}$, $d\geq 1$.
For $\gamma>0$ we define the Lipschitz norm
as
\begin{equation}\label{decayLip33}
\begin{aligned}
\bral A\brar^{\gamma,\mathcal{O}}_{\beta,s}&:=
\bral A\brar^{sup,\mathcal{O}}_{\beta,s}+\gamma
\bral A\brar^{lip,\mathcal{O}}_{\beta,s}
\\&=
\sup_{\omega\in \mathcal{O}}\bral A(\omega)\brar_{\beta,s}+
\gamma
\sup_{\substack{\omega_1,\omega_2\in \mathcal{O}\\ \omega_1\neq\omega_2}}\frac{\bral A(\omega_1)-A(\omega_2)\brar_{\beta,s}}{|\omega_1-\omega_2|}\,.
\end{aligned}
\end{equation}
We denote by $\mathcal{M}^{\gamma,\mathcal{O}}_{\beta,s}$ 
the space of 
families of matrices $A(\omega)$
with finite  $\bral \cdot\brar_{\beta,s}^{\gamma,\mathcal{O}}$-norm.

\end{definition}

For  properties of
 matrices in $\mathcal{M}^{\gamma,\mathcal{O}}_{\beta,s}$ we refer to Appendix \ref{techtech} and in particular Lemma \ref{DecayAlg2} stating a tame property for the norm given by \eqref{decayLip33}.

We end this section with the following definition:
\begin{definition}{\bf (Block-diagonal matrices)}\label{bloccodia}
We say that $A(\vphi)$
 is  \emph{block-diagonal}  if  and only if $A_{[k]}^{[k']}(\vphi)=0$ 
 for any $k\neq k'$ and any $\vphi\in \mathbb{T}^{d}$.
\end{definition}

We notice that operators commuting with $K_0$ have matrices that are block-diagonal: let $Z$ be such that
\begin{equation}\label{commutare}
[K_0,Z]=0\,.
\end{equation}
Since $$[H_0,Z]_{[k]}^{[k']}=(\lambda_{k'}-\lambda_{k} )Z_{[k]}^{[k']}\quad \forall k,k'\, ,$$
 condition \eqref{commutare} implies
that the matrix $(Z_{[k]}^{[k']}))_{k,k'\in \mathbb{N}}$ representing the operator $Z$
is block-diagonal  according to Definition \ref{bloccodia}.

\subsection{Link between  pseudo-differential operators and  matrices}

 To a linear operator $R$ we associate its matrix representation still denoted $R$ through the formula
\begin{equation}\label{QA} 
  R_{[k]}^{[k']} = \int_{\mathtt{M}^n} R \Phi_{[k]} \Phi_{[k']} dx .
 \end{equation}
 In the following we  show that the decay norm $\bral \cdot\brar_{\beta,s}$ 
 (see Definitions \ref{decayNorm} and \ref{1smooth}) 
 is well designed to capture the smoothing property.

\begin{lemma}\label{matrixest} Fix $s>(d+n)/2$ and $\beta \geq0$.
Assume that  $R\in\Rc_{\rho,s}$ with $\rho\geq s+\beta+1/2$ and that $R$ is symmetric  then
 $R\in\mathcal{M}_{ \beta,s}$. 
Moreover, there exists a constant $ C= C(s,\rho,\beta)$  such that  
\begin{align}\label{ranma3}
\bral R\brar_{\beta, s} \le C|R|_{\rho,s,s} 
\end{align}
If $R\in \Rc_{\rho,s}^{\gamma, \mathcal{O}}$ 
then the bound \eqref{ranma3} holds 
with the norms $\bral \cdot\brar_{\beta,s,}$, $|\cdot|_{\rho,s,s}$ 
replaced by the norms 
$\bral \cdot\brar_{\beta,s}^{\gamma, \mathcal{O}}$, 
$|\cdot|^{\gamma, \mathcal{O}}_{\rho,s,s}$.
\end{lemma}

\begin{proof} 
We have for $\ell\in\Z^d$
\begin{align*}
||R_{[k]}^{[k']}(\ell)||_{\mathcal{L}(L^2)} 
&=|\langle D^{\rho+s}R(\ell) \Phi_{[k]}, D^{-\rho-s}\Phi_{[k']} \rangle |
\\&\le 
|| D^{\rho+s}R(\ell) \Phi_{[k]} ||_{L^2} || \Phi_{[k']}||_{L^2} \langle k' \rangle^{-\rho-s} 
\\&\le 
\|R(\ell)\|_{\L(H^s,H^{s+\rho})} \| \Phi_{[k]}\|_{H^{s}} \langle k' \rangle^{-\rho-s} 
\\&\le  
\|R(\ell)\|_{\L(H^s,H^{s+\rho})}\langle k \rangle^{s} \langle k' \rangle^{-\rho-s} 
\end{align*}
where we used that, for $s\in\R$ 
(recall \eqref{equiweight}),
\[
\| \Phi_{k,j}\|_{H^{s}}\sim  \| K_0^s\Phi_{k,j}\|_{L^2}
=\lambda_{k}^{s}\sim \langle k\rangle^s\,.
\]
Similarly, since $R$ is symmetric, 
\[
||R_{[k]}^{[k']}(\ell)||_{\mathcal{L}(L^2)} \leq 
\|R(\ell)\|_{\L(H^s,H^{s+\rho})}\langle k' \rangle^{s} \langle k \rangle^{-\rho-s} \,,
\]
therefore we get
\[
||R_{[k]}^{[k']}(\ell)||_{\mathcal{L}(L^2)} \leq 
\min\big(\langle k' \rangle^{s} \langle k \rangle^{-\rho-s},\langle k\rangle^{s} \langle k' \rangle^{-\rho-s}\big)\|R(\ell)\|_{\L(H^s,H^{s+\rho})}\,.
\]
So, by definition, we get using that $\langle h,\ell\rangle\leq \langle \ell\rangle\langle h\rangle$,
\begin{align*}
\bral \mathcal{D}^{\beta} R\brar_{s}^2 &= 
\sum_{h\in \mathbb{N},\ell\in\Z^d}\langle h,\ell\rangle^{2s}
\sup_{|k-k'|=h}\|(\mathcal{D}^{\beta} R)_{[k]}^{[k']}(\ell)\|^{2}_{\mathcal{L}(L^{2})}\\
&\leq \sum_{\ell\in\Z^d}\langle \ell\rangle^{2s}
\|R(\ell)\|_{\L(H^s,H^{s+\rho})} \sum_{h\in \mathbb{N}}\langle h\rangle^{2s}\sup_{|k-k'|=h}   
\langle k\rangle^{2\beta}
\min\big(\langle k' \rangle^{s} \langle k \rangle^{-\rho-s},\langle k\rangle^{s} \langle k' \rangle^{-\rho-s}\big) \\
 &\le 2^{2\rho-2\beta}|R|^2_{\rho,s,s} 
 \sum_{h \in \mathbb{N}}\langle h\rangle^{2s+2\beta-2\rho} 
\end{align*}
where we used that if $|k-k'|=h$ 
then  $\max(|k|,|k'|)\geq h/2$. 
A similar estimates holds true for 
$\bral R\mathcal{D}^{\beta} \brar_{s}$ and thus for
\[
\bral R\brar_{\beta, s} = \bral \mathcal{D}^{\beta} R\brar_{s}
+ \bral R\mathcal{D}^{\beta} \brar_{s}\,.
\]
Following a similar reasoning one gets the Lipschitz bounds.
\end{proof}

\section{Regularization procedure}\label{sec3}

Let us consider $0<\delta<1$, $r>d/2$
and the operator
\begin{equation}\label{starting-point}
\mathcal{F}=\mathcal{F}(\omega):=\omega\cdot\pa_{\vphi}+\ii (\Delta_g + V(\vphi))\,,
\qquad V\in \mathcal{A}_{\delta,r}\,.
\end{equation}
We also assume that the operator $V$ is \emph{self-adjoint}.
Let us define the diophantine set
 $\mathcal{O}_{0}\subseteq [1/2,3/2]^{d}$ by
\begin{equation}\label{zeroMel}
\mathcal{O}_0:=\big\{\omega\in [1/2,3/2]^{d} \,:\, |\omega\cdot l|\geq 
\frac{4\gamma}{|l|^{\tau}}\,, \; \forall \, l\in\mathbb{Z}^{d}\big\}\,,
\qquad \tau:= d+1\,.
\end{equation}

The aim of this section is to prove the following result.

\begin{theorem}{\bf (Regularization)}\label{thm:regolo}
Let  $\rho_0\geq 0$, $0<\delta<1$ and $r_0>d/2$.
There is $r_{*}=r_{*}(\delta,\rho_0,r_0)$ such that, for
 $r>r_*$ and $S\geq s_0>n/2$,
 there exist 
$p=p(S,\rho_0)\geq1$ and $0<\e_{*}=\e_*(S,\rho_0) $ such that the following holds. If
\begin{equation}\label{smalloo}
\gamma^{-1}\Nc_{\delta,r,p}(V)\leq \e_* \,.
\end{equation}
then there is,
for any $\vphi\in \mathbb{T}^{d}$, for any $\om\in\mathcal O_0$,
  a bounded and invertible map $\Phi \in\L(H^{s},H^{s})$ for any $s\in[s_0,S]$  such that
\begin{equation}\label{fine100}
\mathcal{F}_{+}:=\Phi \mathcal{F}\Phi^{-1}
:=\omega\cdot\pa_{\vphi}+\ii(\Delta_g+Z+R)\,,
\end{equation}
where $Z\in \mathcal{A}_{\delta}^{\gamma,\mathcal{O}_0}$ 
is independent of $\vphi$, $Z$ is Hermitian  
and
\begin{equation}\label{barnantes}
[Z,K_0]=0\,,
\end{equation}
 $R(\varphi)$ is a Hermitian $\rho_0$-smoothing operator  
 in $ \mathcal{R}_{\rho_0,r_0}^{\gamma,\mathcal{O}_0}$.
 
 \noindent
Furthermore $Z=Z_{1}+Z_{2}$ with $Z_{1}\in\mathcal{A}_{\delta}$ is independent of $\omega\in \mathcal{O}_0$, and 
$Z_{2}\in\mathcal{A}_{2\delta-1}^{\gamma,\mathcal{O}_0}$.\\
Moreover the following estimates holds: for any $s\in[s_0,S]$ there exits $q=q(s,\rho_0)\geq p$ and $C=C(s,\rho_0)>0$ such that
\begin{align}\label{stimefinali1}
\Nc_{\delta,s}(Z_{1})+\Nc_{2\delta-1,s}^{\gamma,\mathcal{O}_0}(Z_{2})&\leq 
C\Nc_{\delta,r,q}(V)\,,
\\ \label{stimefinali2}
|R|_{\rho_0,r_0,s}^{\gamma,\mathcal{O}_0} &\leq C
\Nc_{\delta,r,q}(V)\,,\\\label{stimefinali3}
\sup_{\vphi\in\T^d}\|\Phi^{\pm1}(\vphi)-{\rm Id}\|_{\L(H^s,H^{s-\delta})}
& \leq C
\Nc_{\delta,r,q}(V)\,,\\\label{stimefinali4}
\sup_{\vphi\in\T^d}\|\Phi^{\pm1}(\vphi)\|_{\L(H^s,H^{s})} &\leq 1+C
\Nc_{\delta,r,q}(V)\,.
\end{align}
\end{theorem}

As explained in the introduction this Theorem will be 
demonstrated by an iterative procedure alternating 
an averaging step according to the periodic flow of $K_0$ 
(section \ref{ave})  and a step of eliminating 
the time dependence of the averaged term (section \ref{eli}). 
The iteration is detailed in section \ref{ite}.
\subsection{Averaging procedure}\label{ave}
For $A\in\A_{m}$, $m\in \mathbb{R}$, 
we denote for $\tau\in[0,2\pi]$
\begin{equation}
\label{Atau}
A(\tau):= e^{-\ii \t K_0} A e^{\ii \t K_0}
\end{equation}
and 
\begin{equation} \label{averA}
\langle A  \rangle 
:=\int_0^{2\pi} A(\tau) d\tau\,,
\end{equation} 
the average of $A$ along the  flow of $K_0$.\\
We notice that $\langle A  \rangle$ belongs to $\A_m$, commutes with $K_0$ 
and that if $A$ is Hermitian then
$\langle A\rangle$ is Hermitian. 
Let $\mathcal{O}\subset\mathcal{O}_0$ (see \eqref{zeroMel})
and consider the operator 
\begin{equation}\label{oper1}
G=\omega\cdot\pa_{\vphi}+\ii M(\vphi)
\qquad
M(\vphi):=  \Delta_g +  W+ A(\vphi) + R( \vphi)
\end{equation}
where $W \in \mathcal{A}^{\gamma, \mathcal{O}}_{ \delta}$, $0<\delta<1$,
is independent of time and commutes with 
$K_0$, $A \in \mathcal{A}^{\gamma, \mathcal{O}}_{ \delta',r}$ 
for some $\delta'\leq \delta $  and 
$R(\varphi)\in \mathcal{R}_{\rho,r}^{\gamma,\mathcal{O}}$
(see Def. \ref{Rrho}).  We also assume that $ M(\vphi)$ is Hermitian 
$\forall\,\vphi\in \mathbb{T}^{d}$.
 
 \begin{lemma} \label{lemmaS}
 Let $r>d/2$, $0<\delta<1$, 
 $\delta'\leq \delta$ there exists $S\in \mathcal{A}^{\gamma, \mathcal{O}}_{ \delta'-1,r}$ 
 such that for any
   $s> n/2$ and $\rho\geq 0$ 
 there exists $p=p(s,\rho)\geq1$, an increasing function of $s$, and $0<\e_{0}=\e_0(s,\rho) $ such that if
\begin{equation}\label{smalloolem}
\gamma^{-1}\Nc_{\delta',r,p}^{\gamma,\mathcal O}(A)\leq \e_0 \,,
\qquad \Nc_{\delta,r,p}^{\gamma,\mathcal O}(W)\leq 1
\end{equation}
 the symplectic change of variable $ \Phi_S= e^{\ii S(\vphi)}$ belongs to $ \L(H^s,H^s)  $  and we have
 \begin{align}
 G^+&:= \Phi_S \circ G\circ \Phi_S^{-1} 
 =  \omega\cdot\pa_{\vphi}+\ii M^{+}(\vphi)\label{nuovoL+} \\
 M^{+}(\vphi)&:=
 \Delta_g +  W+ \langle A(\vphi)\rangle +  A^+(\vphi) 
 +  R^+ (\vphi)\label{nuovoL+2} 
 \end{align}
where $\langle A(\vphi)\rangle$ is defined as in \eqref{averA},
$A^+ \in \mathcal{A}^{\gamma, \mathcal{O}}_{ \delta'-1,r-1}$  and 
$ R^+ \in \mathcal{R}_{\rho,r-1}^{\gamma,\mathcal{O}}$.
The operator $ M^{+}(\vphi)$ is Hermitian $\forall\, \vphi\in \mathbb{T}^{d}$.

\noindent
Moreover  there exists  
$C=C( s,\rho)$ such that 
\begin{align}
\Nc^{\gamma,\mathcal{O}}_{\delta'-1,r,s} (S)
&\le C \Nc^{\gamma,\mathcal{O}}_{\delta',r,p}(A)\,\label{stimaSS}\\
\sup_{\vphi\in \mathbb{T}^{d}} \|\Phi_S^\t(\vphi)\|_{\L(H^s,H^s)}^{\gamma,\mathcal{O}}
 &\le 1
 +C\Nc^{\gamma,\mathcal{O}}_{ \delta',r,p}(A)\,\label{stimaTTTT2}\\
 \sup_{\vphi\in \mathbb{T}^{d}}  \|\Phi_S^{\t}(\vphi)-{\rm Id}\|^{\gamma,\mathcal{O}}_{\L(H^s,H^{s})} 
 &\leq  C
\Nc^{\gamma,\mathcal{O}}_{\delta',r,p}(A)\, 
 \quad \forall \t \in [0,1]\,. \label{stimaTTTT3}
\\
\Nc^{\gamma,\mathcal{O}}_{ \delta'-1,r-1,s}(A^+) 
&\le C \Nc^{\gamma,\mathcal{O}}_{ \delta',r,p}(A)\, 
\label{stimaSS2}
\\
|R^{+}|_{\rho,r-1,s}^{\gamma,\mathcal{O}}
&\le  C  |R|^{\gamma,\mathcal{O}}_{\rho,r,s} 
+C\Nc^{\gamma,\mathcal{O}}_{\s, \delta',p}(A)\,.\label{stimaSS3}
\end{align}
\end{lemma}

\begin{proof}
The idea comes from \cite{Wein}, \cite{Colin}  
and were extensively used in \cite{BGMR2}. It consists to average with respect to the flow of $K_0$ (see \eqref{averA}) which is periodic since its spectrum is included in $\N +\lambda$ (see \eqref{K0H0}).\\
Let us define $ Y= \frac{1}{2 \pi}\int_0^{2\pi} \t (A - \langle A\rangle )(\t) d\t $. 
Then $Y\in\A_{\delta',r}^{\gamma,\mathcal{O}}$ and by integration by parts we verify that $Y$ solves the homological equation
\begin{equation}\label{omo11}
i [ K_0, Y] = A- \langle A \rangle\,.
\end{equation}
Then we define
\begin{equation}\label{def:S}
 S = \frac{1}{4} (Y K_0^{-1} + K_0^{-1}Y)
 \end{equation}
and we note that $S\in\A_{\delta'-1,r}^{\gamma,\mathcal{O}}$ is a pseudo-differential operator of order $\delta' -1\leq 0$.
Moreover, by using Lemma \ref{propertiespdo}, we deduce the estimate
\eqref{stimaSS}.
By applying Lemma \ref{lemmaB1}
we obtain estimates 
\eqref{stimaTTTT2} and \eqref{stimaTTTT3} (see \eqref{flusso11}, \eqref{flusso12}). 
By an explicit computation we also get 
\begin{equation}\label{omo12}
i[K_0^2 , S] = A- \langle A \rangle - \frac{1}{4} \big[ [A, K_0], K_0^{-1}\big]\, .
\end{equation}

To study the conjugate of $L$ in \eqref{oper1} 
under the map $\Phi_{S}$ defined
as in \eqref{operTotale2} with $S$ in \eqref{def:S} 
we use the Lie expansions
\eqref{operTotale4} and \eqref{operTotale5} for some 
$q\in \mathbb{N}$ large to be chosen later.
Recalling the splitting \eqref{K0H0}-\eqref{H} we 
have by \eqref{operTotale4}
\begin{align}
\Phi_{S}\circ \ii M\circ\Phi_{S}^{-1}&\stackrel{\eqref{omo12}}{=}
\ii K_0^{2}+\ii Q_0+\ii W+\ii \langle A\rangle+ \frac{\ii}{4} \big[ [A, K_0], K_0^{-1}\big]\\&
-\ii [Q_0+W+A,\ii S]
+\sum_{j=2}^{q}\frac{1}{j!}{\rm ad}_{\ii S}^{j}(\ii \Delta_{g}+\ii W+\ii A)\\&+
\frac{1}{q!}\int_{0}^{1}(1-\tau)^{q}e^{\ii \tau S}{\rm ad}^{q+1}_{\ii S}(\ii \Delta_{g}
+\ii W+\ii A)e^{-\ii \tau S}d\tau\\
&+\ii\Phi_{S}\circ R\circ\Phi_{S}^{-1}\,.
\end{align}
Taking into account the
 time contribution  given by \eqref{operTotale5} we obtain 
 that the conjugate $\Phi_{S}\circ G\circ \Phi_{S}^{-1}$ has the form
 \eqref{nuovoL+}-\eqref{nuovoL+2}
where 
\begin{equation}\label{nuovoV}
\begin{aligned}
\ii A^+&=  \frac{\ii}{4} \big[[A, K_0], K_0^{-1}\big] 
 -\ii [Q_0+W+A, \ii S]\\ &
 + \sum_{j=2}^q \frac{1}{j!}{\rm ad}_{\ii S}^{j}(\ii \Delta_{g}+\ii W+\ii A)
-\sum_{p=1}^{q}\frac{1}{p!}{\rm ad}_{\ii S}^{p-1}(\ii  \omega\cdot\pa_{\vphi}S)
\end{aligned}
\end{equation}
and 
\begin{equation}\label{nuovoR}
\begin{aligned}
\ii R^+&=
\frac{1}{q!}\int_{0}^{1}(1-\tau)^{q}e^{\ii \tau S}{\rm ad}^{q+1}_{\ii S}(\ii \Delta_{g}
+\ii W+\ii A)e^{-\ii \tau S}d\tau\\
&+
\frac{1}{q!}\int_{0}^{1}(1-\tau)^{q}\Phi_{S}^{\tau}
{\rm ad}_{\ii S}^{q}(\ii  \omega\cdot\pa_{\vphi}S)\Phi_{S}^{-\tau}d\tau\,,
\\&+\ii \Phi_{S}\circ R\circ\Phi_{S}^{-1}
\end{aligned}
\end{equation}
We need to prove the bounds \eqref{stimaSS2}-\eqref{stimaSS3}.
We start by studying the remainder $R^{+}$ in \eqref{nuovoR}. To simplify the notation
we shall 
write $a \lesssim b$ to denote $a \leq Cb$ for some constant 
$C = C(s,\rho).$

Using 
the smallness condition \eqref{smalloolem},
we have that the 
third summand in \eqref{nuovoR} is a
$\rho$-smoothing operator satisfying \eqref{stimaSS3} by Lemma \ref{lemmaB3}.

By items $(ii), (iii)$ 
of Lemma \ref{propertiespdo} we have (up to smoothing remainder and for some $p$ depending on $s$ and $\rho$)
\begin{align}
\mathcal{N}^{\gamma,\mathcal{O}}_{\delta',r,s}\big({\rm ad}_{\ii S}(\ii \Delta_{g}+\ii W+\ii A)\big)
&+\mathcal{N}^{\gamma,\mathcal{O}}_{\delta'-1,r-1,s}(\omega\cdot\pa_{\vphi}S)\nonumber\\ &
\stackrel{\eqref{smalloolem}, \eqref{stimaSS}}{\lesssim}
\mathcal{N}^{\gamma,\mathcal{O}}_{\delta',r,p}(A)\,.\label{lekain}
\end{align}
By iterating the estimate above and using the smallness condition \eqref{smalloolem} we deduce,
for $1\leq j\leq q$ and for some $p$ depending on $s$, $\rho$, $q$,
\begin{align}
\mathcal{N}^{\gamma,\mathcal{O}}_{j \delta'-2(j-1),r,s}
\big({\rm ad}^{j}_{\ii S}(\ii \Delta_{g}+\ii W+\ii A)\big)&+
\mathcal{N}^{\gamma,\mathcal{O}}_{(j+1)\delta'-1-2j,r-1,s}
({\rm ad}_{\ii S}^{j}(\omega\cdot\pa_{\vphi}S))\nonumber\\
\lesssim \mathcal{N}^{\gamma,\mathcal{O}}_{\delta',r,p}(A)\,.
\label{lekain2}
\end{align}
The sequences $j \delta'-2(j-1)$ and $(j+1)\delta'-1-2j$ are decreasing since
$\delta'\leq 1$. Hence, by choosing $q$ large enough, the integrands in \eqref{nuovoR}
are $\rho$-smoothing operator (with arbitrary $\rho$)
conjugated by the flow $e^{\ii \tau S}$. Therefore 
by  Lemma \ref{lemmaB3} all the expressions in \eqref{nuovoR} are  smoothing remainders
satisfying \eqref{stimaSS3} for some $p$ depending on $s$ and $\rho$.

Let us now consider the terms in \eqref{nuovoV}.
First of all we have
\begin{align*}
\mathcal{N}^{\gamma,\mathcal{O}}_{ \delta'-2,r,s}(\big[[A, K_0], K_0^{-1}\big]) 
&\lesssim
\mathcal{N}^{\gamma,\mathcal{O}}_{  \delta',r, N_1}([A, K_0])
\mathcal{N}^{\gamma,\mathcal{O}}_{-1,N_1}(K_0^{-1}) \\
&\lesssim    \mathcal{N}^{\gamma,\mathcal{O}_0}_{ \delta',r,N}(A)
\mathcal{N}^{\gamma,\mathcal{O}}_{1,N}(K_0)
\mathcal{N}^{\gamma,\mathcal{O}}_{ -1,N}(K_0^{-1})\\
&\lesssim
\mathcal{N}^{\gamma,\mathcal{O}_0}_{\delta',r,N}(A)\,,
\end{align*}
for some constant $N_1\leq N\leq p$ depending only on $s,\rho$.
In the same way (recalling also \eqref{smalloolem}) we have
\begin{align*}
\mathcal{N}^{\gamma,\mathcal{O}}_{\delta+\delta'-2,r,s} ([Q_0+W+A, \ii S])
&\lesssim
\mathcal{N}_{\delta'-1,r,p}^{\gamma,\mathcal{O}}(S)\big(\mathcal{N}_{0,p}(Q_0)
+\mathcal{N}_{\delta,p}^{\gamma,\mathcal{O}}(W)\big)\\
&+\mathcal{N}_{\delta'-1,r,p}^{\gamma,\mathcal{O}}(S)
\mathcal{N}^{\gamma,\mathcal{O}}_{\delta',r,p}(A)
\\&\stackrel{\eqref{smalloolem}, \eqref{stimaSS}}{\lesssim}
\mathcal{N}_{\delta',r,p}^{\gamma,\mathcal{O}}(A)\,.
\end{align*}
The other summands in \eqref{nuovoV} can be estimated 
by using \eqref{lekain} and \eqref{lekain2}.
This proves the \eqref{stimaSS2}.
\end{proof}

\subsection{Time elimination}\label{eli}
Let us consider the operator $L^{+}$ in \eqref{nuovoL+}-\eqref{nuovoL+2} 
obtained after an average step (see Lemma \ref{lemmaS}).
The aim of this section is to eliminate the time dependence (i.e. the dependence with respect to $\vphi$) in the term
$ \langle A(\vphi)\rangle$ in \eqref{nuovoL+2}.
First we introduce the pseudo-differential operator $T=T(\vphi)$ defined as
\begin{align} \label{defT}
T (\vphi) = \sum_{ 0\neq l \in \mathbb{Z}^d}
 \frac{e^{\ii l\cdot\vphi }}{ \ii \omega \cdot l} \langle A(l) \rangle\,.
\end{align}
We have the following Lemma.
\begin{lemma}\label{omoequT}
Let $r\geq 5d/2+9/2$ 
and $\om\in\mathcal O_0$ (see \eqref{zeroMel}). 
Then
the operator $T$ in \eqref{defT} belongs to 
$\mathcal{A}^{\gamma, \mathcal{O}_0}_{\delta',r-(2\tau+1)}$ 
is Hermitian, commutes with the operator $K_0$. Moreover it  solves the equation
\begin{equation}\label{omeoequaTT}
 \langle A(\vphi) \rangle -  \omega \cdot \partial_{\vphi} T 
 =  \langle A(0) \rangle \,,
\end{equation}
and satisfies 
\begin{equation}\label{stimaTT}
\mathcal{N}^{\gamma,\mathcal{O}_0}_{\delta',r-(2\tau+1),s}(T)
\leq C\mathcal{N}^{\gamma,\mathcal{O}_0}_{\delta',r,p}(A)\,.
\end{equation}
Furthermore, setting $\Phi_{T}^{\tau}:=e^{\ii \tau T(\vphi)}$,
we have that for any $s>d/2$ there are
constants $C,p$ (depending only on $s$ and $\rho$ ) such that if \eqref{smalloolem} holds then
\begin{align}
\sup_{\vphi\in \mathbb{T}^{d}} \|\Phi_T^\t\|_{\L(H^s,H^s)}
 &\le 1
 +C\Nc^{\gamma,\mathcal{O}_0}_{ \delta',r,p}(A)\,
 \label{stimaTT2}\\
    \sup_{\vphi\in \mathbb{T}^{d}}  \|\Phi_T^{\t}-{\rm Id}\|_{\L(H^s,H^{s-\delta'})} &\leq  C
\Nc^{\gamma,\mathcal{O}_0}_{ \delta',r,p}(A)\, 
 \quad \forall \t \in [0,1]\,. \label{stimaTT3}
\end{align}
\end{lemma}

\begin{proof}
The operator $T$ is Hermitian and commutes with $K_0$ thanks to the properties of 
$\langle A\rangle$.
The fact that $T$ solves \eqref{omeoequaTT} is obtained by an explicit computation.
The bound \eqref{stimaTT} follows by item $(iii)$ of Lemma \ref{propertiespdo}.
Finally applying  Lemma \ref{lemmaB1} we obtain
 the estimates \eqref{stimaTT2}-\eqref{stimaTT3} (see \eqref{flusso1} and \eqref{flusso2}).
\end{proof}

In the following lemma we study how the operator $G^{+}$ in  \eqref{nuovoL+}-\eqref{nuovoL+2} 
changes under the map $\Phi_{T}^{\tau}$ defined by Lemma \ref{omoequT}.
We have to distinguish the 
cases  $\delta'$ strictly  positive or $\delta'$ less or equal zero.

\begin{lemma}\label{lemmaT} Let  $\delta'\leq 0$ and $r>2\tau+2+d/2$. Let us define
$\delta_{1}:=\delta+\delta'-1$ and $\Phi_{T}:=\Phi_{T}^{1}$. Then the conjugated operator
$G_{1}:=\Phi_{T}\circ G^{+}\circ\Phi_{T}^{-1}$  has the form
\begin{align}
 G_{1}&=  \omega\cdot\pa_{\vphi}+\ii M_1(\vphi)\label{nuovoL++} \\
 M_1(\vphi)&:=
 \Delta_g +  W_1 + A_1(\vphi) 
 +  R_1 (\vphi)\label{nuovoL++2} 
 \end{align}
where 
\begin{equation}\label{nuovoL++3}
W_1=W+ \int_{\mathbb{T}^{d}}\langle A(\vphi)\rangle d\vphi\,,
\end{equation}
is independent of $ \vphi\in \mathbb{T}^{d}$,
$A_1 \in \mathcal{A}^{\gamma, \mathcal{O}}_{\delta_{1},r-2\tau-2}$  
and 
$ R_1 \in \mathcal{R}_{\rho,r-2\tau-2}^{\gamma,\mathcal{O}}$.
The operator $M_1(\vphi)$ is Hermitian $\forall\, \vphi\in \mathbb{T}^{d}$.

\noindent
Moreover for any $s>d/2$ there exist $p=p(s,\rho)$ and  
$C=C( s,\rho)$ such that if \eqref{smalloolem} holds then
\begin{align}
\Nc^{\gamma,\mathcal{O}}_{ \delta_1,r-2\tau-2,s}(A_1) 
&\le C \Nc^{\gamma,\mathcal{O}}_{\delta',r,p}(A)\, 
\label{stimaTTT2}
\\
|R_1|_{\rho,r-2\tau-2,s}^{\gamma,\mathcal{O}}
&\le  C  |R|_{\rho,r,s}^{\gamma,\mathcal{O}}  +
   \Nc^{\gamma,\mathcal{O}}_{ \delta',r,p}(A)\,.     \label{stimaTTT3}
  \end{align}
\end{lemma}

\begin{proof}
Notice that, since $T$
in \eqref{defT} commutes with $K_0$, then $\Phi_{T}\circ K_0^{2}\circ\Phi_{T}^{-1}=K_0^{2}$.
By using the expansions \eqref{operTotale4}-\eqref{operTotale5}
and since $T$ solves \eqref{omeoequaTT}  we have that the conjugate $L_1$ has
the form \eqref{nuovoL++}-\eqref{nuovoL++2} with $W_1$ as in \eqref{nuovoL++3}
and 
\begin{equation}\label{nuovoA+}
\begin{aligned}
\ii A_1&:=\ii A^{+}+\sum_{j=1}^{q}\frac{1}{j!}{\rm ad}^{j}_{\ii T}\big(\ii Q_0+ \ii W
+\ii \langle A(\vphi)\rangle
+\ii A^{+} \big)\\
&-
\sum_{j=2}^{q}\frac{1}{j!}{\rm ad}_{\ii T}^{j-1}(\ii  \omega\cdot\pa_{\vphi}T)\,,
\end{aligned}
\end{equation}
\begin{equation}\label{nuovoR+}
\begin{aligned}
\ii R_1&:=\ii \Phi_{T}\circ R^{+}\circ\Phi_{T}^{-1}\\
&+
\frac{1}{q!}\int_{0}^{1}(1-\tau)^{q}\Phi_{T}^{\tau}{\rm ad}_{\ii T}^{q+1}
\big(\ii Q_0+ \ii W+\ii \langle A(\vphi)\rangle+\ii A^{+} \big)
\Phi_{T}^{-\tau}d\tau\\
&+
\frac{1}{q!}\int_{0}^{1}(1-\tau)^{q}\Phi_{T}^{\tau}{\rm ad}_{\ii T}^{q}
(\ii  \omega\cdot\pa_{\vphi}T)\Phi_{T}^{-\tau}d\tau\,,
\end{aligned}
\end{equation}
and where $q\in \mathbb{N}$ is a large constant to be chosen later.
We now  estimate the different terms in \eqref{nuovoA+}, \eqref{nuovoR+}.

By \eqref{algSeminorm2}
 we have for some $p'=p'(s,\rho)$  
\[
\mathcal{N}^{\gamma,\mathcal{O}_0}_{2\delta'-1,r-2\tau-2,s}
({\rm ad}_{\ii T}(\ii \omega\cdot\pa_{\vphi}T))
\lesssim
\mathcal{N}^{\gamma,\mathcal{O}_0}_{\delta',r-2\tau-2,p'}(T)
\mathcal{N}^{\gamma,\mathcal{O}_0}_{\delta',r-2\tau-2,p'}
(\omega\cdot\pa_{\vphi}T)\,.
\]
On the other hand we have by \eqref{estOmega}
\[
\mathcal{N}^{\gamma,\mathcal{O}_0}_{\delta',r-2\tau-2,p'}
(\omega\cdot\pa_{\vphi}T)
{\lesssim}
\mathcal{N}^{\gamma,\mathcal{O}_0}_{\delta',r-2\tau-1,p'}(T)\,,
\]
thus using \eqref{stimaTT} we deduce
\begin{align}
\mathcal{N}^{\gamma,\mathcal{O}_0}_{2\delta'-1,r-2\tau-2,s}
({\rm ad}_{\ii T}(\ii \omega\cdot\pa_{\vphi}T))&{\lesssim}
\big(\mathcal{N}^{\gamma,\mathcal{O}_0}_{\delta',r,p'}(A)\big)^2
\stackrel{\eqref{smalloolem}}{\lesssim}
\mathcal{N}^{\gamma,\mathcal{O}_0}_{\delta',r,p'}(A)\,.\label{barcoin1}
\end{align}
Similarly we prove
\begin{equation}\label{barcoin2}
\begin{aligned}
\mathcal{N}^{\gamma,\mathcal{O}_0}_{\delta'-1,r-2\tau-2,s}&({\rm ad}_{\ii T}(Q_0))+
\mathcal{N}^{\gamma,\mathcal{O}_0}_{\delta+\delta',r-2\tau-2,s}
({\rm ad}_{\ii T}(W+\langle A(\vphi)\rangle))\\
&+
\mathcal{N}^{\gamma,\mathcal{O}_0}_{2\delta'-2,r-2\tau-2,s}({\rm ad}_{\ii T}(A^{+}))
\lesssim 
\mathcal{N}^{\gamma,\mathcal{O}_0}_{\delta',r,p'}(A)\,.
\end{aligned}
\end{equation}
Notice that, since $0<\delta<1$, the highest order pseudo-differential operator among the ones estimated in \eqref{barcoin1}, \eqref{barcoin2} is the one of order $\delta+\delta'-1<\delta'$. 
By the estimates above, by choosing the constant $q\in \mathbb{N}$ large enough with respect to $\rho$
and by reasoning as in the proof of Lemma \ref{lemmaS} one gets the estimates 
\eqref{stimaTTT2}, \eqref{stimaTTT3}. 
In particular, since $\delta'\leq 0$ we shall use 
Lemma \ref{lemmaB3} in order to estimate the conjugates of smooothing operator under the flow 
$\Phi_{T}^{\tau}$.
\end{proof}

In the next Lemma we study the case in which the generator $T$ of Lemma \ref{omoequT}
has order $\delta'>0$.

\begin{lemma}\label{lemmaTpositivo} 
Let  $0<\delta' \leq \delta$. 
Let us define
$\delta_{1}:=\delta+\delta'-1$ and $\Phi_{T}:=\Phi_{T}^{1}$. 
Fix moreover $r_1>d/2$ and $\rho_1\geq 0$ and assume $r>\max(r_1+d/2,2\tau+2+d/2)$ and $\rho\geq \rho_1+\delta' r_1+1$.
Then the conjugated operator
$G_{1}:=\Phi_{T}\circ G^{+}\circ\Phi_{T}^{-1}$  (see \eqref{nuovoL+}) 
has the form
\eqref{nuovoL++}, \eqref{nuovoL++2}, \eqref{nuovoL++3},
is independent of $ \vphi\in \mathbb{T}^{d}$,
$A_1 \in \mathcal{A}^{\gamma, \mathcal{O}}_{\delta_{1},r-2\tau-2}$  
and 
$ R_1 \in \mathcal{R}_{\rho_1,r_1}^{\gamma,\mathcal{O}}$.
The operator $M_1(\vphi)$ is Hermitian $\forall\, \vphi\in \mathbb{T}^{d}$.

\noindent
Moreover for any $s\in\R$ there exist $p=p(s,\rho)$ and  if \eqref{smalloolem} holds then
$C=C( s,\rho)$ such that 
\begin{align}
\Nc^{\gamma,\mathcal{O}}_{ \delta_1,r-2\tau-2,s}(A_1) 
&\le C \Nc^{\gamma,\mathcal{O}}_{\delta',r,p}(A)\, 
\label{stimaTTT2unbb}
\\
|R_1|_{\rho_1,r_1,s}^{\gamma,\mathcal{O}}
&\le  C  |R|_{\rho,r,s}^{\gamma,\mathcal{O}}  +
   \Nc^{\gamma,\mathcal{O}}_{ \delta',r,p}(A)\,.     \label{stimaTTT3unbb}
  \end{align}
\end{lemma}

\begin{proof}
One reasons as in the proof of Lemma \ref{lemmaT}.
The difference is in estimating the remainder $R_1$ in 
\eqref{nuovoR+}. Since the generator $T$ is of order $\delta'>0$
one has to apply Lemma \ref{lemmaB2}  (instead of Lemma \ref{lemmaB3})
which provides estimates \eqref{stimaTTT3unbb} isntead of the \eqref{stimaTTT3}.
\end{proof}

\subsection{Proof of Theorem \ref{thm:regolo}}\label{ite}
In this section we give the proof of 
Theorem \ref{thm:regolo} 
which is based on an iterative application of Lemmata of the previous section. 
Recalling \eqref{starting-point} we set
\[
G_0:=\mathcal{F}=\omega\cdot\pa_{\vphi}+\ii \Delta_{g}+\ii V\,.
\]
The operator $G_0$ above has the form \eqref{oper1} with
\begin{equation}\label{paramt0}
\mathcal{O}=\mathcal{O}_0,\quad W\equiv0\,,\quad R\equiv0\,,\quad A(\vphi)=V(\vphi)\,,\quad \delta'=\delta\,.
\end{equation}
Since $V$ is $C^\infty$,  $r>d/2$ can be chosen arbitrary large. We will chose it later in function of the order $\delta$, of the final regularity $r_0$ and the smoothness $\rho_0$ prescribed by \eqref{stimefinali2}.\\
Lemma \ref{lemmaS} prodides $p_1 (S)$ such that 
if $p\geq p_1(S)$ in \eqref{smalloo} then \eqref{smalloolem} holds for any $s\in[s_0,S]$. By  applying Lemma \ref{lemmaS} to $G_0$ we obtain a symplectic map
$\Phi_{S_0}$  such that (see \eqref{nuovoL+})
\begin{equation}\label{tripoli25}
\begin{aligned}
\widetilde{G}_0&:= \Phi_{S_0}\circ G_0\circ \Phi_{S_0}^{-1}=
\omega\cdot\pa_{\vphi}+\ii \Delta_{g}+\ii \langle V(\vphi)\rangle+\ii \widetilde{A}_{0}
+\ii \widetilde{R}_{0}\,,\\
&\widetilde{A}_{0} \in \mathcal{A}^{\gamma, \mathcal{O}}_{ \delta-1,r-1}\,, \qquad
 \widetilde{R}_{0} \in \mathcal{R}_{\rho,r-1}^{\gamma,\mathcal{O}}\,,
\end{aligned}
\end{equation}
with $\rho>0$ arbitrary to be chosen later
and where $\langle V(\vphi)\rangle$ is defined as in \eqref{averA}.
We apply Lemma \ref{lemmaTpositivo} to the operator given by \eqref{tripoli25}
with $\rho_1\rightsquigarrow \rho_0$ of Theorem \ref{thm:regolo}
and $r_1>d/2$ (to be chosen later)
provided that $\rho$ and $r$ are sufficiently large ($\rho>\rho_0+\delta r_1+1$ and $r>\max(r_1+d/2,2\tau+2+d/2)$).
Hence we obtain a symplectic map $\Phi_{T_0}$
such that
\[
\begin{aligned}
G_1&:=\Phi_{T_0}\circ \widetilde{G}_0\circ\Phi_{T_0}^{-1} 
=\Phi_{T_0}\circ\Phi_{S_0}\circ G_0\circ \Phi_{S_0}^{-1}\circ\Phi_{T_0}^{-1}\\
&=
\omega\cdot\pa_{\vphi}+\ii \Delta_{g}+\ii W_{1}+\ii A_{1}+\ii R_{1}
\end{aligned}
\]
with $W_{1}:=\int_{\mathbb{T}^{d}}\langle V(\vphi)\rangle d\vphi$, 
$A_{1}\in \mathcal{A}^{\gamma,\mathcal{O}_0}_{2\delta-1,r_1}$, 
$R_1\in \mathcal{R}^{\gamma,\mathcal{O}_0}_{\rho_0,r_1}$ and
estimates \eqref{stimaTTT2unbb} \eqref{stimaTTT3unbb} are satisfied for all $s\in[s_0,S]$ provided $p\geq p_2(p_1,S)$  depending on $p_1$ and $S$ (and still increasing in $S$).\\
We notice that $W_1$ is independent of $\vphi$ 
and of the parameters $\omega\in\mathcal{O}_0$.

Now we want to iterate this procedure.

Let us first consider the case\footnote{Actually Theorem \ref{thm:regolo} will be applied only in this case.} $0<\delta\leq 1/2$. Then  $2\delta-1\leq0$ and hence, form now on, 
we will apply iteratively 
Lemmata \ref{lemmaS} and \ref{lemmaT} (instead of Lemma \ref{lemmaTpositivo}).

We introduce the following parameters:
for $n\geq1$ we set
\begin{equation}\label{paramRegu}
 \delta_n = (n+1)\delta -n  \,,\qquad 
 r_n =r_1-n(2\tau+2) \,, \qquad q_n=q_0\circ q_{n-1}\,
 \end{equation}
 where $q_0(\cdot)=q_1(\cdot,S)$ is the composition of the two function $s\mapsto p(s)$ given by  Lemmata \ref{lemmaS} and \ref{lemmaT} and $q_1=p_2\circ p_1$. We notice that $q_n$ is an increasinf function of $S$.\\
 Then applying Lemmata \ref{lemmaS} and \ref{lemmaT} iteratively, 
 there exist symplectic changes of variables 
 $\{\Phi_{S_n}\}_n$ and $\{\Phi_{T_n}\}_n$ such that, 
 setting $\Phi_{n}:=\Phi_{T_n}\circ \Phi_{S_n}$, we have 
 \begin{equation}\label{iteroJam}
G_{n+1}:= \Phi_n\circ G_n \circ\Phi_n^{-1}
= \omega\cdot\pa_{\vphi}+\ii \Delta_{g} +\ii  W_{n+1} 
+ A_{n+1}  + R_{n+1} 
\end{equation}
 where $W_n$ is pseudo-differential operator 
 independent of $\vphi$ of order $\delta$ commuting with $K_0;$ 
 $A_n$ is pseudo-differential operator of order 
 $\delta_n;$  $R_n$ is $\rho_0-$smoothing operator. 
 Moreover, 
 by estimates \eqref{stimaSS2}, \eqref{stimaSS3} 
 and \eqref{stimaTTT2}, \eqref{stimaTTT3} we get
 \begin{align}\label{iteroJam2}
 \mathcal{N}^{\gamma,\mathcal{O}_0}_{  \delta,s}(W_{n}) +
\mathcal{N}^{\gamma,\mathcal{O}_0}_{  \delta_n,r_n,s}(A_{n}) +
|R_{n}|^{\gamma,\mathcal{O}_0}_{\rho_{0},r_n,s}
&\le C \mathcal{N}_{  \delta,r,q_n}(V)\quad \text{ for all }s\in[s_0,S]\,.
\end{align}
We perform $N=N(\rho_0,\delta)$ steps of this procedure 
in order to get $\delta_{N}=\delta -N(1-\delta)\leq-\rho_0$. 
This require to choose $r_1$ (and hence $r$) sufficiently large.
More precisely, we want $r_N\geq r_0$, the prescribed regularity, and thus in view of \eqref{paramRegu}
$r_{1}\geq N(2\tau+2)+r_0$. Then recalling that we need $r>\max(r_1+d/2,2\tau+2+d/2)$ we have to chose
$$r>\max(N(2\tau+2)+r_0+d/2,2\tau+2+d/2):=r_{*}(\delta,\rho_0,r_0)\,.$$
Moreover the constant $\rho$ appearing in \eqref{tripoli25}
should be chosen in such a way
\[
\rho\geq \rho_0+\delta r_1+1\geq \rho_0+2N(\tau+1)+r_0\,.
\]
%
Therefore the operator $G_{N}$, defined as in \eqref{iteroJam}, 
has the form \eqref{fine100} with $Z:=W_N$. 
We notice that 
$W_1:=\int_{\mathbb{T}^{d}}\langle V(\vphi)\rangle d\vphi\in\A_\delta$ 
is independent of $\om$ and that 
$W_N-W_1 \in\mathcal{A}^{\mathcal{O}_0}_{2\delta-1,r_1}$ 
which leads to the desired splitting $Z=Z_1+Z_2$. 
The bounds \eqref{stimefinali1}, \eqref{stimefinali2} 
follows by \eqref{iteroJam2} with $q=q_N$. 
The estimates \eqref{stimefinali3}, \eqref{stimefinali4} 
follows by composition and estimates 
\eqref{stimaTTTT2}, \eqref{stimaTTTT3}, 
\eqref{stimaTT2} and  \eqref{stimaTT3}.

The case $1/2\leq\delta<1$ requires to apply  Lemmata \ref{lemmaS} and \ref{lemmaTpositivo} iteratively to construct $\tilde A_n\in \A_{\tilde\delta_n,\tilde r_n}$ with $\delta_n=2\delta_{n-1}-1$ and $\delta_0=\delta$, until $\tilde\delta_n$ became negative. Then we can apply the second procedure using Lemmata \ref{lemmaS} and \ref{lemmaT} as in the previous case.

\section{KAM reducibility}\label{sec4}

In this section we will prove an abstract KAM Theorem for a matrix operator of the form
 \begin{equation}\label{start}
L_0=L_0(\omega;\vphi):=\omega\cdot\pa_{\vphi}+
\ii (\Delta_g+Z_0+R_0(\vphi))\,.
\end{equation}
To precise our hypothesis on $L_0$ we define the following constants
\begin{equation}\label{costanteB}
\begin{aligned}
&\mathtt{b}:= 6d+15n+23\,,\qquad \tau=d+1\,,\qquad \rho=5n+3\,,\\
&\gamma\in(0,1)\,,\quad 0\leq\ka\leq1\,\qquad  s_0>\frac{d+n}{2}\,,\quad S\geq s_0+\mathtt{b}\,.
\end{aligned}
\end{equation}

In this section we assume:
\begin{itemize}
\item[({\bf A1})] the matrix $Z_0$ is Hermitian, block diagonal, 
independent of $\vphi$ and Lipschitz in 
$\om\in \O\subseteq{\mathcal{O}_0}\equiv \mathcal{O}_0(\gamma,\tau)$ 
(see \eqref{zeroMel}).
Furthermore, denoting 
$(\mu^{(0)}_{k,j})_{j=1,\cdots,d_k}$ the eigenvalues 
of the block $(Z_0)_{[k]}^{[k]}$, 
we assume that there exists  $\ka\geq0$ 
such that (recall that $c_0$ is defined in \eqref{spec})
\begin{align}\label{mu1}
 |\mu^{(0)}_{k,j}(\om)|&\leq \frac {c_0}2 |k|,\quad \om\in\O,\ k\in\N,\ j=1,\cdots,d_k\,,\\
 \label{mu2}
 \|(Z_0)_{[k]}^{[k]}\|_{\L(L^2)}^{lip,\O}&\leq \frac14\langle k\rangle^{-\ka}\,,
 \quad  k\in\N\,.
\end{align}
\item[({\bf A2})] the operator $R_0$ is in 
$\mathcal{M}^{\gamma,\mathcal{O}}_{\rho,S}$  (see Def. \ref{1smooth})
and
is Hermitian.
\end{itemize}
Let us define
\begin{equation}\label{paramPiccolez}
\epsilon:=\gamma^{-1}
\bral R_0\brar^{\gamma,\mathcal{O}_0}_{\rho,s_0+\mathtt{b}}\,.
\end{equation}
We shall prove the following.

\begin{theorem}{\bf (Reducibility)}\label{thm:redu}
Let $s\in [s_0,S-\mathtt{b}]$.
There exist positive constants $\epsilon_0=\epsilon_0(s),C=C(s)$ 
such that, if
\begin{equation}\label{piccolopiccolo}
\epsilon\leq \epsilon_0\,,
\end{equation}
then there is a set $\mathcal{O}_{\epsilon}\subseteq\mathcal{O}$ with 
\begin{equation}\label{misura}
{\rm meas}(\mathcal{O}\setminus\mathcal{O}_{\epsilon})\leq C\gamma
\end{equation}
such that the following holds. For any $\omega\in \mathcal{O}_{\epsilon}$ there are
\begin{itemize}
\item[$(i)$]{\bf (Normal form)} a  matrix $Z_{\infty}=Z_0+\tilde Z_\infty$ 
with $\tilde Z_\infty\in \mathcal{M}_{\rho,s}^{\gamma,\mathcal{O}_{\epsilon}}$ 
which is $\vphi$-independent, Hermitian and block-diagonal;

\item[$(ii)$]{\bf (Conjugacy)} a bounded and  invertible  map 
$\Phi_{\infty}=\Phi_{\infty}(\omega,\vphi) : H^s
\to H^s$ 
such that for all $\vphi\in\T^d$, for all $\om\in\O_\epsilon$,
\begin{equation}\label{operatoreFinale}
L_{\infty}:=\Phi_{\infty}L_0\Phi_{\infty}^{-1}:=\omega\cdot\pa_{\vphi}+\ii (\Delta_g+Z_{\infty})\,.
\end{equation}
\end{itemize}
Moreover  we have
\begin{equation}\label{stimaMappain}
\sup_{\vphi\in \mathbb{T}^{d}}\|\Phi_{\infty}^{\pm1}(\vphi)-{\rm Id}\|_{\L( H^s;H^{s})}
\leq C
\gamma^{-1}\bral R_0\brar_{\rho,s+\mathtt{b}}^{\gamma,\mathcal{O}}\,, 
\quad \forall\,  \om\in\O_\epsilon\,,
\end{equation}
\begin{equation}\label{stimaZinfty}
\bral  \tilde Z_\infty\brar_{\rho,s}^{\gamma,\mathcal{O}_{\epsilon}}\leq
C
\bral R_0\brar_{\rho,s+\mathtt{b}}^{\gamma,\mathcal{O}}\,.
\end{equation}
\end{theorem}

\subsection{The KAM step}

The proof of Theorem \ref{thm:redu} is based on an iterative scheme. In this section we
show how to perform one step of the iteration.
We consider an operator 
\begin{equation}\label{opLpartenza}
L:=\omega\cdot\pa_{\vphi}+\ii(\Delta_g+Z+R)\,,
\end{equation}
where $Z=Z_{0}+Z_{2}$ is Hermitian with 
$Z_{0}$ satisfying  ({\bf A1}) and 
$Z_{2}\in \mathcal{M}_{\rho,s}^{\gamma,\mathcal{O}}$ 
for all $s\in[s_0, S]$ and for some 
$\mathcal{O}\subseteq\mathcal{O}_{0}$ (see \eqref{zeroMel}). 
The remainder $R$ satisfies ({\bf A2}), i.e. belongs to  $\mathcal{M}_{\rho,s}^{\gamma,\mathcal{O}}$ 
for all $s\in[s_0, S]$ and is Hermitian.

\subsubsection{Control of the small divisors}
Let us denote by $\mu_{k,j}$, $k\in \mathbb{N}$ and $j=1,\ldots, d_{k}$ (see \eqref{dimension}), 
the eigenvalues 
of the block $(\Delta_g+Z)_{[k]}^{[k]}$.
First of all we prove the following.
\begin{lemma}\label{lem:Lipeign}
One has
\begin{equation}\label{lipboundEigen}
\sup_{k\in\mathbb{N}}\langle k\rangle^{\ka}|\mu_{[k]}|^{lip,\mathcal{O}}\leq \frac14+
\bral Z_{2}\brar^{lip,\mathcal{O}}_{\ka,s_0}\,.
\end{equation}
\end{lemma}

\begin{proof}
By Corollary $A.7$ in \cite{FG1} the Lipschitz variation of the eigenvalues of an Hermitian matrix
is controlled by the Lipschitz variation of the matrix.
Then, in view of hypothesis ({\bf A1}), we get
$$|\mu_{[k]}|^{lip,\mathcal{O}_0}\leq  \|(Z_0)_{[k]}^{[k]}\|_{\L(L^2)}^{lip,\O}+\|(Z_{2})_{[k]}^{[k]}\|_{\L(L^2)}^{lip,\mathcal{O}_0}\leq \langle k\rangle^{-\ka}(\frac14+\bral Z_{2}\brar^{lip,\mathcal{O}_0}_{\ka,s_0})$$
and the \eqref{lipboundEigen} follows.
\end{proof}

We define 
the set $\mathcal{O}_{+}\subseteq \mathcal{O}$ 
of parameters $\o$ for which we 
have a good control of the small divisors. 
We set, for $N\geq 1$,
\begin{equation}\label{calO+}
\begin{aligned} 
\mathcal{O}_{+}\equiv\mathcal{O}_{+}(\g,N)&:=\Big\{\omega\in \mathcal{O}\, : \, 
|\omega\cdot l+\mu_{k,j}
- \mu_{k',j'}|
\geq \frac{2\gamma}{N^{\tau}\langle k,k'\rangle^{2n+2}}\,,\\
&\qquad\quad |l|\leq N,\quad k,k'\in\N\,, \quad j=1,\ldots,d_k\,, \\
&\qquad \qquad\quad\quad   j'=1,\ldots, d_k' \,,
\quad (l,k,k')\neq (0,k,k)
\Big\}\,.
\end{aligned}
\end{equation}
We have the following.
\begin{lemma}\label{lem:misuro} Assume that 
\begin{equation}\label{uffauffa}
\bral {Z}_{2}\brar_{\ka,s_0+\mathtt{b}}^{\gamma,\mathcal{O}}\leq \gamma/8
\end{equation}
for some $0<\gamma\leq \frac{c_0}5$ (see \eqref{spec}) then 
we have
\begin{equation}\label{misuro}
{\rm meas}\big(\mathcal{O}\setminus \mathcal{O}_{+}(\g,N)\big)\leq 
C\gamma N^{-1}
\end{equation}
for some constant $C>0$ depending only on $d$.
\end{lemma}

\begin{proof}
We write
\begin{equation*}
\mathcal{O}\setminus \mathcal{O}_{+}=\bigcup_{\substack{l\in \mathbb{Z}^{d}
, |l|\leq N \\ k,k'\in \mathbb{N}\\ (\ell,k,k')\neq(0,k,k)}}\bigcup_{\substack{j=1,\ldots,d_{k}\\
j'=1,\ldots,d_{k'}}}R_{l,k,k'}^{j,j'}
\end{equation*}
where
\begin{equation*}
R_{l,k,k'}^{j,j'}:=\Big\{\omega\in \mathcal{O}\, : \, 
|\omega\cdot l+\mu_{k,j}
-\mu_{k',j'}|\leq\frac{2\gamma}{N^{\tau}\langle k,k'\rangle^{2n+2}}
\Big\}.
\end{equation*}

Notice that when $l=0$ and $k\neq k'$ then $R_{l,k,k'}^{j,j'}=  \emptyset$ for all $j,j'$. 
Indeed in  such case
 we get using \eqref{mu1}, \eqref{spec} and \eqref{uffauffa}
 $$|\lambda_{k}+\mu_{k,j}
- \lambda_{k'}+\mu_{k',j'}|
\geq \frac{c_0}{2}(k+k')-2\bral {Z}_{2}\brar_{\ka,s_0+\mathtt{b}}^{\infty,\mathcal{O}} \geq  \frac{c_0}{2}-\frac\gamma4\geq   2\gamma.$$
Let us now consider the case $l\neq0$. 
We give the  estimate of the measure of
a single \emph{bad} set 
$R_{l,k,k'}^{j,j'}$. Let us consider the Lipschitz  function
$$f(\o)=\omega\cdot l+\mu_{k,j}(\o)
- \mu_{k',j'}(\o)=\om\cdot l+g(\om)\,.$$ 
Using condition \eqref{uffauffa} we have that Lemma \ref{lem:Lipeign} 
implies that (recall that $l\neq0$)
 \[
 |g|^{lip,\mathcal O}\leq \frac12\,.
 \]
Then
Lemma $5.2$ in \cite{FG1} implies that
${\rm meas}(R_{l,k,k'}^{j,j'}) \leq \frac{C\gamma}{N^{\tau}\langle k,k'\rangle^{2n+2}}$ for some constant $C>0$ depending only on $d$.
Finally by \eqref{dimension}
we have that 
\[
d_{k}d_{k'}\leq \langle k ,k'\rangle^{2(n-1)}\,.
\]
Hence
\begin{equation*}
\begin{aligned}
{\rm meas}\big(\mathcal{O}\setminus\mathcal{O}_{+}\big)&\leq C
\sum_{\substack{l\in \mathbb{Z}^{d}
, 0<|l|\leq N\\ k,k'\in \mathbb{N}}}\sum_{\substack{j=1,\ldots,d_{k}\\
j'=1,\ldots,d_{k'}}}R_{l,k,k'}^{j,j'}\\&
{\leq }\ C
\sum_{\substack{l\in \mathbb{Z}^{d}
, 0<|l|\leq N\\ k,k'\in \mathbb{N}}}2\frac{\g}{N^{\tau}\langle k,k'\rangle^{4}} 
\leq    
CN^{-1}\gamma
\end{aligned}
\end{equation*}
since $\tau=d+1$.
\end{proof}

\subsubsection{Resolution of the Homological equation}
In this section we solve the following homological equation equation 
\begin{equation}\label{omoeq}
-\ii\omega\cdot\pa_{\vphi}S+\big[\ii S,\Delta_g+Z\big]+R={\rm Diag}R+Q
\end{equation}
where $Q$ is some remainder to be determined and 
\begin{equation}\label{newNorm}
\begin{aligned}
&{\rm Diag}R=\big(({\rm Diag}R)_{[k]}^{[k']}(l)
\big)_{l\in\mathbb{Z}^{d},k,k'\in\mathbb{N}}\,,\\
&({\rm Diag}R)_{[k]}^{[k']}(l):=0\, \;\;\text{ for }\;\; l\neq0\,, k,k'\in \mathbb{N}\;\;
{\rm or }\;\; l=0\,, k\neq k'\,,\\
&({\rm Diag}R)_{[k]}^{[k]}(0):= A_{[k]}^{[k]}(0)\,, \qquad {\rm otherwise}\,.
\end{aligned}
\end{equation}

 \begin{lemma}{\bf (Homological equation)}\label{omoequation} Let 
 $R\in \mathcal{M}^{\gamma,\mathcal{O}}_{\rho,s}$ for $s\in[s_0,S]$, 
 $\rho$ in \eqref{costanteB}.
 For any $\omega\in \mathcal{O}_+\equiv  \mathcal{O}_+(\gamma,N) $ 
 (defined in \eqref{calO+}) 
 there exist Hermitian matrices 
$S,Q$ solving equation \eqref{omoeq} and
 satisfying
 \begin{equation}\label{gene}
 \begin{aligned}
  \bral S\brar_{s}^{\gamma,\mathcal{O}_+} &\leq_{s}  
\frac{N^{2\tau+1}}{\gamma}
\bral R \brar_{\rho,s}^{\gamma,\mathcal{O}}\,, \\
 \bral\mathcal{D}^{\pm \rho}S\mathcal{D}^{\mp\rho}\brar_{s}^{\gamma,\mathcal{O}_+}
 &\leq_{s}  
\frac{N^{2\tau+\rho+1}}{\gamma}
\bral R \brar_{\rho,s}^{\gamma,\mathcal{O}}\,,
 \end{aligned}\qquad s\in [s_0,S]\,,
 \end{equation}
 \begin{equation}\label{ultrastima}
 \begin{aligned}
 \bral Q\brar_{\rho,s}^{\gamma,\mathcal{O}_+}&\leq_{s}  
\bral R \brar_{\rho,s+\mathtt{b}}^{\gamma,\mathcal{O}} 
N^{-\mathtt{b}}\,,\\
\bral Q\brar_{\rho,s+\mathtt{b}}^{\gamma,\mathcal{O}_+}&\leq_{s}  
\bral R \brar_{\rho,s+\mathtt{b}}^{\gamma,\mathcal{O}}\,,
 \end{aligned}\qquad
 s\in[s_0,S-\mathtt{b}]\,.
 \end{equation}
%
 \end{lemma}

\begin{proof}
For $N>0$ we define  (recall \eqref{timeMat})
the matrix $\Pi_{N}R$ as
\begin{equation}\label{cutOffLem}
(\Pi_{N}R)_{[k]}^{[k']}(l):=\left\{
\begin{aligned} &R_{[k]}^{[k']}(l)\,,\quad l\in \mathbb{Z}^{d}\,, k,k'\in\mathbb{N}\,,\quad
\begin{aligned}
&|l|\leq N\,, \\ &|k-k'|\leq N\,,\end{aligned}
\\
&0\,,\qquad {\rm otherwise}
\end{aligned}
\right.
\end{equation}

Then we set 
\begin{equation}\label{restoUltra}
Q=(1-\Pi_{N})R\end{equation}
By Lemma \ref{DecayAlg2}, and since the regularity in $\vphi$ has been fixed at $r=\mathtt b$, one deduces the estimates \eqref{ultrastima}.
Moreover, recalling \eqref{newNorm},
we have that
 equation \eqref{omoeq} is equivalent to 
\begin{equation}\label{omoeq2}
\mathcal{G}(l,k,k',\omega)S_{[k]}^{[k']}(l)+
(\Pi_{N}R)_{[k]}^{[k']}(l)=0
\end{equation}
for any $l\in \mathbb{Z}^{d}$, $k,k'\in \mathbb{N}$ with $(l,k,k')\neq (0,k,k)$
where the operator $\mathcal{G}(l,k,k',\omega)$ is the linear operator
acting on complex $d_{k}\times d_{k'}$-matrices as
\begin{equation}\label{azioneG}
\mathcal{G}(l,k,k',\omega)A:=-\ii \Big[ \omega\cdot l 
+ \big( \Delta_g+Z\big)_{[k]}^{[k]}\Big]
A+\ii
 A
\big( \Delta_g+Z\big)_{[k']}^{[k']}\,.
\end{equation}
Now, since $\big( \Delta+Z\big)_{[k]}^{[k]}$ is Hermitian, there is 
a orthogonal $d_{k}\times d_{k}$-matrix $U_{[k]}$
such that 
\[
U_{[k]}^{T} \big( \Delta_g+Z\big)_{[k]}^{[k]} U_{[k]}
=D_{[k]}:={\rm diag}_{j=1,\ldots,d_{k}}\big(\mu_{k,j}\big)\,,
\]
where $\mu_{k,j}$ 
are the eigenvalues of the $k$-th block.
By setting
\begin{equation*}
\widehat{S}_{[k]}^{[k']}(l):=U_{[k]}^{T} S_{[k]}^{[k']}(l)U_{[k']}\,,
\qquad 
\widehat{R}_{[k]}^{[k']}(l):=U_{[k]}^{T} R_{[k]}^{[k']}(l)U_{[k']}
\end{equation*}
equation \eqref{omoeq2}
reads
\begin{equation}\label{omoeq3}
-\ii \Big( \omega\cdot l + D_{[k]}\Big)\widehat{S}_{[k]}^{[k']}(l)+\ii
\widehat{S}_{[k]}^{[k']}(l) D_{[k']}+(\Pi_{N}\widehat{R})_{[k]}^{[k']}(l)=0\,.
\end{equation}
For $\omega\in \mathcal{O}_{+}$ (see \eqref{calO+})
the solution of \eqref{omoeq3} is
 given by (recalling the notation \eqref{notazioneBlock})
\begin{equation}\label{solOmoeq}
\widehat{S}_{k,j}^{k',j'}(l):=\left\{\begin{aligned}
& \frac{ -\ii\widehat{R}_{k,j}^{k',j'}(l)}{\omega\cdot l
+\mu_{k,j}-\mu_{k',j,}}\,,
\qquad \begin{aligned}&|l|\leq N\,, \\ 
&\; |k-k'|\leq N\,,\end{aligned}\;\; (l,k,k')\neq (0,k,k)\,,\\
& 0\,, \qquad \qquad\qquad {\rm otherwise}\,.
\end{aligned}\right.
\end{equation}
Since $R$ is Hermitian 
it is easy to check that
also $S$ is Hermitian. Using the bound on the small divisors in \eqref{calO+}
we have that
\begin{equation}\label{stimaInfty}
|\widehat{S}_{k,j}^{k',j'}(l)|\leq \gamma^{-1}|\widehat{R}_{k,j}^{k',j'}(l)| N^{\tau}\langle k,k'\rangle^{2n+2}\,.
\end{equation}
Then, by denoting by $\|\cdot\|_{\infty}$ the sup-norm of a $d_{k}\times d_{k'}$-matrix,
we deduce
\begin{equation}\label{stimaInfty2}
\begin{aligned}
\|{S}_{[k]}^{[k']}(l)\|_{\L(L^2)}&=\|\widehat{S}_{[k]}^{[k']}(l)\|_{\L(L^2)}\leq \sqrt{d_{k}d_{k'}}\|\widehat{S}_{[k]}^{[k']}(l)\|_{\infty}
\\
&\stackrel{\eqref{stimaInfty}, \eqref{dimension}}{\leq}
\gamma^{-1}\|{R}_{[k]}^{[k']}(l)\|N^{\tau}\langle k,k'\rangle^{3n+1}\,.
\end{aligned}
\end{equation}
We now estimates the decay norm of the matrix 
$S$. We have 
\begin{equation}\label{tripoli2}
\begin{aligned}
\bral S\brar^{2}_{s}&\stackrel{\eqref{stimaInfty2}}{\leq} \gamma^{-2}N^{2\tau}
\sum_{l,h}\langle l,h\rangle^{2s}\sup_{|k-k'|=h}\|R_{[k]}^{[k']}(l)\|^{2}
\langle k,k'\rangle^{6n+2}\\
&\leq 
\gamma^{-2}N^{2\tau}
\sum_{l,h}\langle l, h\rangle^{2s}\sup_{\substack{|k-k'|=h \\ k\geq k'}
}\|(\mathcal{D}^{\rho}R)_{[k]}^{[k']}(l)\|^{2}
\langle k\rangle^{(6n+2-2\rho)}\\
&+
\gamma^{-2}N^{2\tau}
\sum_{l,h}\langle l, h\rangle^{2s}\sup_{\substack{|k-k'|=h \\ k< k'}
}\|(R\mathcal{D}^{\rho})_{[k]}^{[k']}(l)\|^{2}
\langle k'\rangle^{(6n+2-2\rho)}
\\&
\le_{s}\gamma^{-2}N^{2\tau}
\bral R\brar^2_{\rho,s}\,,
\end{aligned}
\end{equation}
provided that $\rho\geq 3n+1 $
which is true thanks to the choices in \eqref{costanteB}. 
Hence the bound \eqref{beebee2} in Lemma \ref{DecayAlg}
implies 
\begin{equation}\label{stimaInfty3}
\bral\mathcal{D}^{\pm\rho}S\mathcal{D}^{\mp\rho}\brar_{s}\le_{s}\gamma^{-1}N^{\tau+\rho}
\bral R\brar_{\rho,s}\,.
\end{equation}
To obtain \eqref{gene}, it remains to estimate the Lipschitz variation of the matrix $S$.
We reason as in the proof of item $(iii)$ of Lemma \ref{propertiespdo}.
To simplify the notation, for any $l\in \mathbb{Z}^{d}$, $k,k'\in \mathbb{n}$, $j=1,\ldots,d_{k}$
and $j'=1,\ldots,d_{k'}$, we set
\begin{equation}\label{tripoli}
d(\omega):=\ii (\omega\cdot l+\mu_{k,j}(\omega)-\mu_{k',j'}(\omega))\,,
\quad \forall \omega\in \mathcal{O}_{+}\,.
\end{equation}
By \eqref{solOmoeq} we have that, for any $\omega_1,\omega_2\in \mathcal{O}_{+}$
\[
\begin{aligned}
\widehat{S}_{k,j}^{k',j'}(\omega_1;l)-\widehat{S}_{k,j}^{k',j'}(\omega_2;l)&=
\frac{ \widehat{R}_{k,j}^{k'j'}(\omega_1;l)- \widehat{R}_{k,j}^{k'j'}(\omega_2;l)}{
d(\omega_1)}\\
&+\frac{d({\omega_1})-d({\omega_2})}{d(\omega_1)d(\omega_2)}\widehat{R}_{k,j}^{k',j'}(\omega_2;l)\,.
\end{aligned}
\]
Using the \eqref{lipboundEigen}, \eqref{mu2} we deduce
\[
\frac{|d(\omega_1)-d(\omega_2)|}{|\omega_{1}-\omega_2|}\lesssim |l|\,, \qquad \forall \omega_1,\omega_2\in \mathcal{O}_+\,,\;\; \omega_1\neq\omega_2\,.
\]
Therefore, recalling \eqref{calO+}, \eqref{suplip}
and reasoning as in \eqref{stimaInfty}, \eqref{stimaInfty2}, 
we get
\[
\begin{aligned}
\|{S}_{[k]}^{[k']}(l)\|^{lip,\mathcal{O}_{+}}_{\mathcal{L}(L^{2})}&\lesssim
\gamma^{-1}N^{\tau}\langle k,k'\rangle^{3n+1}\|R_{[k]}^{[k']}(l)\|^{lip,\mathcal{O}}\\
&+\gamma^{-2}N^{2\tau+1}\langle k,k'\rangle^{5n+3}\|R_{[k]}^{[k']}(l)\|^{sup,\mathcal{O}}\,.
\end{aligned}
\]
Finally, reasoning as in \eqref{tripoli2} and using \eqref{decayNorm2}, we deduce
\begin{equation}\label{tripoli3}
\bral S\brar^{lip,\mathcal{O}_{+}}_{s}\le_s \gamma^{-1} N^{\tau}
\bral R\brar_{\rho,s}^{lip,\mathcal{O}}+\gamma^{-2}N^{2\tau+1}
\bral R\brar_{\rho,s}^{sup,\mathcal{O}}\,,
\end{equation}
provided that $\rho\geq 5n+3$, which is true by \eqref{costanteB}.
Combining \eqref{tripoli2} and \eqref{tripoli3} (recall \eqref{decayLip})
we get the first bound in \eqref{gene}. The second one follows 
by  \eqref{beebee2} in Lemma \ref{DecayAlg}. 
\end{proof}

\begin{lemma}\label{stimaMappa}
There is $C(s)>0$ (depending only on $s\geq s_0$) such that, if
\begin{equation}\label{smallsmall}
\gamma^{-1} C(s)N^{2\tau+1}
\bral R\brar_{\rho,s_0}^{\gamma,\mathcal{O}}\leq \frac{1}{2}\,,
\end{equation}
then  the map 
$\Phi=e^{\ii S}={\rm Id}+\Psi$, with $S$ given by Lemma \ref{omoequation},
satisfies 
\begin{equation}\label{stimaMappa1}
\bral\Psi\brar_{s}^{\gamma,\mathcal{O}_+}\leq_s 
\gamma^{-1} N^{2\tau+1}
\bral R\brar_{\rho,s}^{\gamma,\mathcal{O}}\,.
\end{equation}
\end{lemma}

\begin{proof}
By \eqref{gene} and \eqref{smallsmall}
we have that
\begin{equation}\label{vulcano}
C(s)\bral S\brar_{s_0}^{\gamma,\mathcal{O}_+}
\leq
1/2\,,
\end{equation}
which implies the  \eqref{mito30}. Hence 
the \eqref{stimaMappa1} follows by Lemma \ref{well-well}.
\end{proof}

\subsubsection{The new remainder}

In this subsection we study the conjugate of the operator $L$
under the map $\Phi$ given by Lemma \ref{stimaMappa}.
We first define the new normal form $Z_{+}$ 
as
\begin{equation}\label{newZ}
Z_{+}:=Z+ {\rm Diag}R\,.
\end{equation}
We have the following.
\begin{lemma}{\bf (New normal form)}\label{newnormal}
We have that $Z_{+}$ in \eqref{newZ} is
$\vphi$-independent, Hermitian and block-diagonal,
and satisfies
\begin{equation}\label{diffNormal}
\bral Z_{+}-Z\brar_{\rho,s}^{\gamma,\mathcal{O}}\le_{s}
\bral R\brar_{\rho,s}^{\gamma,\mathcal{O}}\,.
\end{equation}
\end{lemma}

\begin{proof}
It follows by construction.
\end{proof}

\begin{lemma}{\bf (The new remainder)}\label{nuovoresto}
Assume that the smallness condition \eqref{smallsmall}
holds true.
Then one has
\begin{equation}\label{nuovoOp}
L_{+}:=\Phi L\Phi^{-1}:=\omega\cdot\pa_{\vphi}+\ii(\Delta_g+Z_{+}+R_{+})
\end{equation}
where 
$Z_{+}$ is the normal form given by \eqref{newZ}
and the new remainder $R_{+}$ is Hermitian and satisfies for all $s\in[s_0,S-\mathtt b]$
\begin{equation}\label{nuovoRem1}
\bral R_{+}\brar_{\rho,s}^{\gamma,\mathcal{O}_+}\leq_s 
N^{-\mathtt{b}}\bral R\brar^{\gamma,\mathcal{O}}_{\rho,s+\mathtt{b}}+
\gamma^{-1}N^{2\tau+\rho+1}
\bral R\brar^{\gamma,\mathcal{O}}_{\rho,s_0}
\bral R\brar^{\gamma,\mathcal{O}}_{\rho,s} 
\end{equation}
\begin{equation}\label{nuovoRem1bis}
\bral R_{+}\brar_{\rho,s+\mathtt{b}}^{\gamma,\mathcal{O}_+}\leq_s 
\bral R\brar^{\gamma,\mathcal{O}}_{\rho,s+\mathtt{b}}+
\gamma^{-1}N^{2\tau+\rho+1}
\bral R\brar^{\gamma,\mathcal{O}}_{\rho,s_0}
\bral R\brar^{\gamma,\mathcal{O}}_{\rho,s+\mathtt{b}}\,.
\end{equation}
\end{lemma}

\begin{proof}
Using the Lie expansions \eqref{operTotale4} and \eqref{operTotale5}
we get
\[
\begin{aligned}
L_{+}:=\Phi L\Phi^{-1}&=\omega\cdot\pa_{\vphi}+\ii(\Delta_g+Z)+\ii R+
\ii [\ii S, \Delta_g+Z]-\ii \omega\cdot\pa_{\vphi}S
\\&+\ii \sum_{p\geq 1}\frac{\ii^{p}}{p!}{\rm ad}_{S}^{p}(R)
+\ii\sum_{p\geq2}\frac{\ii^{p-1}}{p!}{\rm ad}_{S}\big(
[\ii S, \Delta_g+Z]- \omega\cdot\pa_{\vphi}S
\big)\,.
\end{aligned}
\]
Hence, 
equations \eqref{omoeq}, \eqref{newZ} 
 lead to the following formula:
$$R_+=Q+\widetilde{R}_{+}$$
with $Q:=(1-\Pi_N)R$ satisfying  \eqref{ultrastima} and 
\begin{equation}\label{gene3}
\widetilde{R}_{+}:=
\sum_{p\geq 2}\frac{\ii^{p-1}}{p!}{\rm ad}_{S}^{p-1}\Big(
{\rm Diag}R+Q-R
\Big) +\sum_{p\geq1}\frac{1}{p!}{\rm ad}_{S}^{p}\big(R\big)\,.
\end{equation}
Thus, in order to prove \eqref{nuovoRem1} we need to estimate $\widetilde{R}_+$.
Consider (for instance) the composition operator $SR$.
In order to control the $\bral\cdot\brar_{\rho,s}^{\gamma,\mathcal{O}_{+}}$-norm
we shall bound the  decay norm of 
$\mathcal{D}^{\rho}S R$. The estimates for 
for $S R\mathcal{D}^{\rho}$ is the same.
We have that
\begin{equation}\label{allergia}
\begin{aligned}
\bral D^{\rho}SR\brar_{s}^{\gamma,\mathcal{O}_{+}}=
\bral D^{\rho}S\mathcal{D}^{-\rho} \mathcal{D}^{\rho}R\brar_{s}^{\gamma,\mathcal{O}_{+}}
&\stackrel{\eqref{decayTame}}{\le_s}
\bral D^{\rho}S\mathcal{D}^{-\rho}\brar_{s}^{\gamma,\mathcal{O}_{+}}
\bral \mathcal{D}^{\rho}R\brar_{s_0}^{\gamma,\mathcal{O}}\\&+
\bral D^{\rho}S\mathcal{D}^{-\rho}\brar_{s_0}^{\gamma,\mathcal{O}_{+}}
\bral\mathcal{D}^{\rho}R\brar_{s}^{\gamma,\mathcal{O}}\\
&\stackrel{\eqref{gene}}{\le_s}
\gamma^{-1}N^{2\tau+\rho+1}\bral R\brar_{\rho,s}^{\gamma,\mathcal{O}}
\bral R\brar_{\rho,s_0}^{\gamma,\mathcal{O}}
\end{aligned}
\end{equation}
The commutator $[S,R]$ satisfies the same bound as \eqref{allergia}. Therefore, 
by \eqref{allergia}, \eqref{ultrastima}, formula \eqref{gene3}, the smallness assumption
\eqref{smallsmall},
and reasoning as in the proof of Lemma \ref{well-well}
we get the \eqref{nuovoRem1} and \eqref{nuovoRem1bis}.
\end{proof}

\subsection{Iteration and Convergence}\label{itteroittero}
In this section we introduce a new constant
\begin{equation}\label{costantine1}
\mathtt{a}:=\mathtt{b}-2=6d+15n+18\,.
\end{equation}
For $N_0\geq1$ we define the sequence $(N_\nu)_{\nu\geq0}$ by
$$ N_{\nu}:=N_{0}^{\chi^{\nu}}\,,\; \nu\geq0$$ 
with $\chi:=3/2$ and we set $N_{-1}=1$.
The proof of 
Theorem \ref{thm:redu} is based on the 
 following iterative lemma.

\begin{proposition}{\bf (Iteration)}\label{ittero} Let $s\in[s_0,S-\mathtt b]$.
There exist  $C(s)>0$ and $N_0\equiv N_0(s)\ge1$ 
such that, if (recall \eqref{paramPiccolez})
\begin{equation}\label{smallcond}
C(s)N_0^{2\tau+1+\rho+\mathtt a}\epsilon\leq \frac12\,,
\end{equation}
then
we may construct recursively sets $\mathcal{O}_{\nu}\subset \mathcal{O}_{0}$ and 
operators, defined for $\omega\in\mathcal{O}_{\nu}$, 
\begin{equation}\label{opNuesimo}
L_{\nu}:=L_{\nu}(\omega):=\omega\cdot\pa_{\vphi}+\ii (\Delta_g+Z_{\nu}+R_{\nu})
\end{equation}
so that the following properties are satisfied 
for all $\nu\in\N$:

\smallskip
\noindent
$({\bf S1})_{\nu}$ 
There is a Lipschitz family of symplectic maps  $\Phi_{\nu}(\vphi)=\Phi_{\nu}(\vphi,\omega):={\rm Id}+\Psi_{\nu}(\vphi)\in\L(H^s,H^s)$
defined on $\mathcal{O}_{\nu}$ such that, for $\nu\geq 1$,
\begin{equation}\label{coniugato}
L_{\nu}:=\Phi_{\nu}L_{\nu-1}\Phi_{\nu}^{-1}\,,
\end{equation}
and, for $s\in [s_0,S-\mathtt b]$,
\begin{equation}\label{stimamappaNu}
\bral \Psi_{\nu}\brar^{\gamma,\mathcal{O}_{\nu}}_{s}
\leq \gamma^{-1}\bral R_0\brar^{\gamma,\mathcal{O}_0}_{\rho,s+\mathtt{b}}
N_{\nu-1}^{2\tau+2}
N_{\nu-2}^{-\mathtt a}\,.
\end{equation}

\smallskip
\noindent
$({\bf S2})_{\nu}$ 
The operator $Z_{\nu}=Z_{0}+Z_{\nu,2}$ where 
$Z_{\nu,2}$ is $\vphi$-independent, block-diagonal and Hermitian.
Moreover it satisfies 
\begin{equation}\label{differenzaNormalform}
\bral Z_{\nu}-Z_{\nu-1}\brar_{\rho,s}^{\gamma,\mathcal{O}_{\nu}}
\leq  
\bral R_0\brar_{\rho,s+\mathtt{b}}^{\gamma,\mathcal{O}_0}N_{\nu-2}^{-\mathtt{a}}\,.
\end{equation}
Moreover there
is a sequence of Lipschitz function
 \[
 \mu_{[k]}^{(\nu)} : \mathcal{O}_{0} \to \mathbb{R}^{d_{k}}\,,\quad k\in \mathbb{N}
 \]
such that, for $\omega\in \mathcal{O}_{\nu}$, the functions
 $\mu^{(\nu)}_{k,j}$, for $j=1,\ldots,d_{k}$, are the eigenvalues
 of the block 
 \[
 (\Delta+Z_{\nu})_{[k]}^{[k]}\,,
 \]
satisfying 
\begin{equation}\label{modo1}
\sup_{k\in \mathbb{N}}\langle k\rangle^{\ka}
 |\mu^{(\nu)}_{[k]} |^{lip,\mathcal{O}_0}\leq \frac14+\epsilon\sum_{j=1}^{\nu-1}2^{-\nu}\,,
 \end{equation}
 where $\epsilon$ is defined in \eqref{paramPiccolez}.

\smallskip
\noindent
$({\bf S3})_{\nu}$ The remainder $R_{\nu}$ is Hermitian and
satisfies, for any $s\in [s_0,S-\mathtt b]$,
\begin{equation}\label{stimaR}
\bral R_{\nu}\brar^{\gamma,\mathcal{O}_{\nu}}_{\rho,s}\leq 
\bral R_0\brar^{\gamma,\mathcal{O}_0}_{\rho,s+\mathtt{b}} N_{\nu-1}^{-\mathtt{a}}\,,
\qquad
\bral R_{\nu}\brar^{\gamma,\mathcal{O}_{\nu}}_{\rho,s+\mathtt{b}}\leq 
\bral R_0\brar^{\gamma,\mathcal{O}_0}_{\rho,s+\mathtt{b}} N_{\nu-1}\,.
\end{equation}

\smallskip
\noindent
$({\bf S4})_{\nu}$ One has $\mathcal{O}_{\nu}\subset \mathcal{O}_{\nu-1}\subset \mathcal{O}_0$ (see \eqref{zeroMel}) and   
\begin{equation}\label{measO}
{\rm meas}\big(\mathcal{O}_{\nu+1}\setminus \mathcal{O}_{\nu}\big)
\leq \gamma N_{\nu}^{-1} \qquad \text{and} \qquad
{\rm meas}\big(\mathcal{O}_{0}\setminus\mathcal{O}_{\nu+1}\big)\leq 2\gamma\,.
\end{equation}
\end{proposition}

\begin{proof}
We proceed by induction. We first verify the inductive step.
So we assume  that conditions $({\bf Si})_{j}$, $i=1,2,3,4$, hold for  $1\leq j\leq \nu$.
We shall prove that they holds for $\nu\rightsquigarrow \nu+1$.

We define the set $\mathcal{O}_{\nu+1}$ as in \eqref{calO+}
with $\mathcal{O}\rightsquigarrow \mathcal{O}_{\nu}$, 
$N\rightsquigarrow N_{\nu}$, $\mu_{k,j}\rightsquigarrow \mu_{k,j}^{(\nu)}$.
Using the \eqref{differenzaNormalform}  for $s\rightsquigarrow s_0$, 
we have that
\begin{equation}\label{after1}
\begin{aligned}
\bral Z_{\nu,2}^{(\nu)}\brar_{\rho,s_0}^{\gamma,\mathcal{O}_{\nu}}&\leq
\bral Z_{0,2}\brar_{\rho,s_0+\mathtt{b}}^{\gamma,\mathcal{O}_0}+
\sum_{j=1}^{\nu}\bral Z_{j,2}-Z_{j-1,2}\brar_{\rho,s_0+\mathtt{b}}^{\gamma,\mathcal{O}_{j}}\\
&
\stackrel{\eqref{differenzaNormalform},\eqref{paramPiccolez}}{\leq}\gamma\epsilon\sum_{j=0}^{\nu-1}2^{-j}
\end{aligned}
\end{equation}
for  $N_0\geq 1$ large enough. Hence condition \eqref{uffauffa} is 
satisfied for $\epsilon$ small enough, i.e.again  $N_0$ large enough (recall \eqref{smallcond}). Therefore Lemma \ref{lem:misuro} implies that
\eqref{measO} holds for the set $\mathcal{O}_{\nu+1}$ which is the $({\bf S4})_{\nu+1}$.

We define the new normal form $Z_{\nu+1}$
as (recall \eqref{newZ})
\[
Z_{\nu+1}:=Z_{\nu}+ {\rm Diag}(R_{\nu})\,.
\]
Lemma \ref{newnormal}, applied with $R\rightsquigarrow R_{\nu}$,
together with the estimates \eqref{stimaR}, implies the estimate \eqref{differenzaNormalform}.
Let $\widetilde{\mu}_{[k]}^{(\nu+1)}$ be the eigenvalues 
of the block $ (\Delta+Z_{\nu+1})_{[k]}^{[k]}$ which are defined on the set $\mathcal{O}_{\nu+1}$.
The bounds \eqref{modo1} follows by Lemma \ref{lem:Lipeign}
and \eqref{after1}.
Moreover, by Kirtzbraun Theorem, there is an extension $\mu_{[k]}^{(\nu+1)}$
of $\widetilde{\mu}_{[k]}^{(\nu+1)}$ to the whole set $\mathcal{O}_0$
with the same Lipschitz norm. This prove the $({\bf S2})_{\nu+1}$.

Then we want to construct a map $\Phi_{\nu+1}={\rm Id}+\Psi_{\nu+1}$.
  First by the inductive hypothesis \eqref{stimaR} we deduce that (with $C(s)$ given in Lemma \ref{stimaMappa})
\begin{equation}\label{after3}
\begin{aligned}
C(s)\gamma^{-1}N_{\nu}^{2\tau+1}\bral R_{\nu}\brar_{\rho,s_0}^{\gamma,\mathcal{O}_{\nu}}
&\leq C(s)\gamma^{-1}N_{\nu}^{2\tau+1}N_{\nu-1}^{-\mathtt{a}}
\bral R_{0}\brar_{\rho,s_0+\mathtt{b}}^{\gamma,\mathcal{O}_{0}}
\stackrel{\eqref{paramPiccolez},\eqref{costantine1}}{\leq}C(s)
\epsilon N_{\nu}^{2\tau+1-\frac{2}{3}\mathtt{a}}\leq \frac{1}{2} 
\end{aligned}
\end{equation}
for $\epsilon$ small enough and 
since
$2\tau+1-\frac{2}{3}\mathtt{a}\leq 0$.
Hence the smallness condition 
\eqref{smallsmall} holds true. We then apply
 Lemmata \ref{omoequation} and \ref{stimaMappa}
with $R\rightsquigarrow R_{\nu}$ and $\mathcal{O}_{+}\rightsquigarrow \mathcal{O}_{\nu+1}$
and construct a map $\Phi_{\nu+1}={\rm Id}+\Psi_{\nu+1}$.\\  
Furthermore 
using 
\eqref{stimaMappa1}, \eqref{stimaR} at rank $\nu$, $N_{\nu-1}=N_\nu^{2/3}$ and ${2\tau+1-\frac23\mathtt{a}}\leq -1$, we obtain the estimate \eqref{stimamappaNu} at rank $\nu+1$.
This proves the $({\bf S1})_{\nu+1}$.

We finally set
\begin{equation}\label{opnu+1}
L_{\nu+1}:=\Phi_{\nu+1}L_{\nu}\Phi_{\nu+1}^{-1}=\omega\cdot\pa_{\vphi}+\ii(\Delta_g+Z_{\nu+1}+R_{\nu+1})
\end{equation}
where the remainder $R_{\nu+1}$ is given by Lemma \ref{nuovoresto}.
We have
\begin{equation}\label{stimaR+1}
\begin{aligned}
\bral R_{\nu+1}\brar^{\gamma,\mathcal{O}_{\nu+1}}_{\rho,s}&\stackrel{\eqref{nuovoRem1}}{\leq_s}
N_{\nu}^{-\mathtt{b}}\bral R_{\nu}\brar^{\gamma,\mathcal{O}_{\nu}}_{\rho,s+\mathtt{b}}+
\gamma^{-1}N_{\nu}^{2\tau+\rho+1}
\bral R_{\nu}\brar^{\gamma,\mathcal{O}_{\nu}}_{\rho,s_0}
\bral R_{\nu}\brar^{\gamma,\mathcal{O}_{\nu}}_{\rho,s}\\
&\stackrel{\eqref{stimaR}}{\le_s}
\bral R_0\brar^{\gamma,\mathcal{O}_0}_{\rho,s+\mathtt{b}}
\Big( N_{\nu}^{-\mathtt{b}+1}+\gamma^{-1}\bral R_0\brar^{\gamma,\mathcal{O}_0}_{\rho,s_0} 
N_{\nu}^{2\tau+\rho+1-\frac{4}{3}\mathtt{a}}
\Big)
\\&
\le N_{\nu}^{-\mathtt{a}}\bral R_0\brar^{\gamma,\mathcal{O}_0}_{\rho,s+\mathtt{b}}
\end{aligned}
\end{equation}
for $N_0$ large enough where we used that $\gamma^{-1}\bral R_0\brar^{\gamma,\mathcal{O}_0}_{\rho,s_0} \leq 1$
(thanks to \eqref{piccolopiccolo})
and 
\[
\mathtt{b}\geq \mathtt{a}+2\,, \qquad 2\tau+\rho+1-\frac{1}{3}\mathtt{a}\leq -1.
\]
The latter condition is implied by the choice of $\mathtt{a}$ in \eqref{costantine1}
recalling the \eqref{costanteB}.
The \eqref{stimaR+1} is the first estimate in \eqref{stimaR} at step $\nu+1$.
We now give the estimate in ``high'' norm.
We have 
\begin{equation}\label{stimaR+2}
\begin{aligned}
\bral R_{\nu+1}\brar^{\gamma,\mathcal{O}_{\nu+1}}_{\rho,s+\mathtt{b}}
&\stackrel{\eqref{nuovoRem1bis}}{\leq_s}
\bral R_{\nu}\brar^{\gamma,\mathcal{O}_{\nu}}_{\rho,s+\mathtt{b}}+
\gamma^{-1}N_{\nu}^{2\tau+\rho+1}
\bral R_{\nu}\brar^{\gamma,\mathcal{O}_{\nu}}_{\rho,s_0}
\bral R_{\nu}\brar^{\gamma,\mathcal{O}_{\nu}}_{\rho,s+\mathtt{b}}\\
&\stackrel{\eqref{stimaR}}{\le_s}
\bral R_0\brar^{\gamma,\mathcal{O}_0}_{\rho,s+\mathtt{b}}N_{\nu-1}
\Big( 1 +\gamma^{-1}\bral R_0\brar^{\gamma,\mathcal{O}_0}_{\rho,s_0+\mathtt b} 
N_{\nu}^{2\tau+\rho+1}N_{\nu-1}^{-\mathtt a-1}
\Big)\\
&\le_{s} N_{\nu} \bral R_0\brar^{\gamma,\mathcal{O}_0}_{\rho,s+\mathtt{b}}
\end{aligned}
\end{equation}
for $N_0$ large enough depending on $s$ and thanks to fact that $3\tau+\frac32\rho+\frac12-\mathtt{a}\leq0$. 
This is the $({\bf S3})_{\nu+1}$.

\vspace{0.5em}

Now we have to verify the initial step: $\nu=1$. $({\bf S2})_1$ and $({\bf S4})_1$ are proved exactly in the way as in the inductive step. Now to proceed  we have to construct $\Phi_1$ but now \eqref{after3} becomes
\begin{equation}\label{after31}
\begin{aligned}
C(s)\gamma^{-1}N_{0}^{2\tau+1}\bral R_{0}\brar_{\rho,s_0}^{\gamma,\mathcal{O}_{0}}
\stackrel{\eqref{paramPiccolez}}{\leq}C(s)
\epsilon N_{0}^{2\tau+1}\leq \frac{1}{2} 
\end{aligned}
\end{equation}
which is less than $\frac12$ for $\epsilon$ and $N_0$ satisfying \eqref{smallcond}.

Furthermore 
using 
\eqref{stimaMappa1} we obtain 
$$\bral\Psi_1\brar_{s}^{\gamma,\mathcal{O}_1}\leq C(s) 
\gamma^{-1} N_0^{2\tau+1}
\bral R_0\brar_{\rho,s}^{\gamma,\mathcal{O}}\leq \gamma^{-1} N_0^{2\tau+2}
\bral R_0\brar_{\rho,s}^{\gamma,\mathcal{O}}$$
for $N_0$ large enough.
This proves the $({\bf S1})_{\nu+1}$.

Then we set
\begin{equation}\label{opnu+11}
L_{1}:=\Phi_{1}L_{0}\Phi_{1}^{-1}=\omega\cdot\pa_{\vphi}+\ii(\Delta_g+Z_{1}+R_{1})
\end{equation}
where the remainder $R_{1}$ is given by Lemma \ref{nuovoresto}.
We have
\begin{equation}\label{stima+1}
\begin{aligned}
\bral R_{1}\brar^{\gamma,\mathcal{O}_{1}}_{\rho,s}&\stackrel{\eqref{nuovoRem1}}{\leq_s}
N_{0}^{-\mathtt{b}}\bral R_{0}\brar^{\gamma,\mathcal{O}_{0}}_{\rho,s+\mathtt{b}}+
\gamma^{-1}N_{0}^{2\tau+\rho+1}
\bral R_{0}\brar^{\gamma,\mathcal{O}_{0}}_{\rho,s_0}
\bral R_{0}\brar^{\gamma,\mathcal{O}_{0}}_{\rho,s}\\
&\stackrel{\eqref{stimaR}}{\le_s}
\bral R_0\brar^{\gamma,\mathcal{O}_0}_{\rho,s+\mathtt{b}}
\Big( N_{0}^{-\mathtt{b}+1}+\epsilon
N_{0}^{2\tau+\rho+1}
\Big)
\\&
\le N_{0}^{-\mathtt{a}}\bral R_0\brar^{\gamma,\mathcal{O}_0}_{\rho,s+\mathtt{b}}
\end{aligned}
\end{equation}
for $N_0$ large enough where we used 
\eqref{smallcond}
and 
$
\mathtt{b}\geq \mathtt{a}+2\,.
$
The \eqref{stima+1} is the first estimate in \eqref{stimaR} at step $1$, the other is proved similarly.

\end{proof}

\begin{proof}[{\bf Proof of Theorem \ref{thm:redu}}]
Consider the operator $L_0$ in \eqref{start}. The smallness condition \eqref{piccolopiccolo}
implies the \eqref{smallcond}, hence 
Proposition \ref{ittero} applies. We define the set 
\begin{equation}\label{finalset}
\mathcal O_\epsilon\equiv\mathcal{O}_{\infty}:=\cap_{\nu\geq0}\mathcal{O}_{\nu}\,.
\end{equation}
By the measure estimate \eqref{measO} we deduce \eqref{misura}.
For any $\omega\in \mathcal{O}_{\infty}$, $\nu\geq0$,
we define 
(see \eqref{coniugato}, \eqref{stimamappaNu})
 the map
\begin{equation}\label{compoMondo}
\begin{aligned}
\widetilde{\Phi}_{\nu+1}&:=\Phi_{1}\circ\Phi_{2}\circ\cdots
\Phi_{\nu+1}=\widetilde{\Phi}_{\nu}\Phi_{\nu+1}=
\widetilde{\Phi}_{\nu}({\rm Id}+\Psi_{\nu+1})\,.
\end{aligned}
\end{equation}
We want to prove that $(\widetilde{\Phi}_\nu)_{\nu\geq1}$ converges in $\mathcal M_s^{\g,\mathcal{O}_{\infty}}$.  Let us define
\begin{equation}\label{defiTame}
\delta^{(\nu)}_{s}:=\bral\widetilde{\Phi}_{\nu}\brar_{s}^{\gamma,\mathcal{O}_{\infty}}\,.
\end{equation}
We have
\begin{equation}\label{tametame}
\delta^{(\nu+1)}_{s_0}\stackrel{\eqref{decayTame}}{\leq}
\delta^{(\nu)}_{s_0}(1+C
\bral{\Psi}_{\nu+1}\brar^{\gamma,\mathcal{O}_{\infty}}_{s_0})
\stackrel{\eqref{stimamappaNu},\eqref{paramPiccolez}}{\leq}\delta^{(\nu)}_{s_0}(1+C\epsilon N_\nu^{-1})\,.
\end{equation}
By iterating the \eqref{tametame} we get, for any $\nu$,
\begin{equation}\label{tametame2}
\bral\widetilde{\Phi}_{\nu}\brar_{s_0}^{\gamma,\mathcal{O}_{\infty}}\leq 
(1+\bral\Psi_{1}\brar_{s_0}^{\gamma,\mathcal{O}_{\infty}})\Pi_{j\geq1}(1+C\epsilon N_\nu^{-1})\leq 2
\end{equation}
where we used the \eqref{stimamappaNu} to estimate 
$\bral\Psi_{1}\brar_{s_0}^{\gamma,\mathcal{O}_{\infty}}$ and we take $N_0$ large enough. \\
The high norm of $\widetilde{\Phi}_{\nu+1}$ is estimated by 
\begin{equation}\label{tametame3}
\begin{aligned}
\delta_{s}^{(\nu+1)}&\stackrel{\eqref{decayTame}}{\leq}
\delta^{(\nu)}_{s}(1+C(s)
\bral{\Psi}_{\nu+1}\brar^{\gamma,\mathcal{O}_{\infty}}_{s_0})+C(s)
\bral{\Psi}_{\nu+1}\brar^{\gamma,\mathcal{O}_{\infty}}_{s}\bral\widetilde{\Phi}_{\nu}\brar_{s_0}^{\gamma,\mathcal{O}_{\infty}}\\
&\stackrel{\eqref{stimamappaNu},\eqref{tametame2}}{\leq}
\delta^{(\nu)}_{s}(1+C(s)\epsilon N_\nu^{-1})+\e_{\nu}
\end{aligned}
\end{equation}
where
\[
\e_{\nu}:= C(s)
\gamma^{-1}\bral R_0\brar^{\gamma,\mathcal{O}_0}_{\rho,s+\mathtt{b}}N_{\nu}^{-1}\,.
\]
By iterating \eqref{tametame3}, using $\Pi_{j\geq0}(1+C(s)\epsilon N_\nu^{-1})\leq 2$ for $N_0$ large enough, we obtain
\begin{equation}\label{tametame4}
\bral\widetilde{\Phi}_{\nu}\brar_{s}^{\gamma,\mathcal{O}_{\infty}}\leq \bral\widetilde{\Phi}_{1}\brar_{s}^{\gamma,\mathcal{O}_{\infty}}+2\sum_{j\geq1} \e_j\leq 1+C(s)\gamma^{-1}
\bral R_0\brar_{\rho,s+\mathtt{b}}^{\gamma,\mathcal{O}_0}\,.
\end{equation}
Then we have
\begin{align}\nonumber
\bral\widetilde{\Phi}_{\nu+1}-&\widetilde{\Phi}_{\nu}\brar^{\gamma,\mathcal{O}_{\infty}}_{s}
=\bral\widetilde{\Phi}_{\nu}\Psi_{\nu+1}\brar_{s}^{\gamma,\mathcal{O}_{\infty}}\\\nonumber
&\stackrel{\eqref{decayTame}}{\le_s} \bral\widetilde{\Phi}_{\nu}\brar_{s}^{\gamma,\mathcal{O}_{\infty}}\bral\Psi_{\nu+1}\brar_{s_0}^{\gamma,\mathcal{O}_{\infty}}+\bral\widetilde{\Phi}_{\nu}\brar_{s_0}^{\gamma,\mathcal{O}_{\infty}}\bral\Psi_{\nu+1}\brar_{s}^{\gamma,\mathcal{O}_{\infty}}\\ \nonumber
&\stackrel{ \eqref{tametame2}, \eqref{tametame4}, \eqref{stimamappaNu}}{\le_s}
(1+\gamma^{-1}\bral R_0\brar_{\rho,s+\mathtt{b}}^{\gamma,\mathcal{O}_0})\epsilon N_{\nu}^{-1} +\gamma^{-1}\bral R_0\brar_{\rho,s+\mathtt{b}}^{\gamma,\mathcal{O}_0}N_{\nu}^{-1}\\ \label{tametame5}
&\le_s\gamma^{-1}\bral R_0\brar_{\rho,s+\mathtt{b}}^{\gamma,\mathcal{O}_0}N_{\nu}^{-1}\,.
\end{align}
Now fix $s\in[s_0,S-\mathtt b]$, since by hypothesis ({\bf A2}), $R_0\in \mathcal{M}^{\gamma,\mathcal{O}}_{\rho,s+\mathtt{b}} $, we deduce from the last estimate that $(\widetilde{\Psi}_{\nu})_{\nu\geq0}$
is a Cauchy sequence in $\mathcal M_s^{\g,\mathcal{O}_{\infty}}$.
Hence $\widetilde{\Phi}_{\nu}{\to}\Phi_{\infty}\in\mathcal M_s^{\g,\mathcal{O}_{\infty}}$. Furthermore writing
$$\|\widetilde{\Phi}_{\nu}-{\rm Id}\|_{\L(H^s,H^s)}\leq \sum_{j=2}^\nu \|\widetilde{\Phi}_j-\widetilde{\Phi}_{j-1}\|_{\L(H^s,H^s)} +\|\Psi_1\|_{\L(H^s,H^s)} $$
we deduce by \eqref{tametame5}  and Lemma \ref{decayspaziotempo} that 
 $\Phi_{\infty}$ satisfies \eqref{stimaMappain}.
The estimate on $\Phi^{-1}_{\infty}-{\rm Id}$ follows
by using Neumann series and reasoning as in the proof of Lemma
\ref{well-well}.
By \eqref{differenzaNormalform} we deduce that
$Z_{\nu,2}$ is a Cauchy sequence in
$\mathcal M_{\rho,s}^{\g,\mathcal{O}_{\infty}}$. Hence we set
\begin{equation}\label{tametame6}
Z_{\infty}=Z_{0}+Z_{\infty,2}:=Z_{0}+
\lim_{\nu\to\infty}Z_{\nu,2}\,,
\end{equation}
The
\eqref{stimaZinfty} follows again by \eqref{differenzaNormalform}. We also notice that \eqref{stimaR} implies that $R_\nu \to 0$ in $ \mathcal{M}^{\gamma,\mathcal{O}_\infty}_{\rho,s} $.
Now by applying iteratively the \eqref{coniugato}
we have that $L_{\nu}=\widetilde{\Phi}_{\nu}L_{0}\widetilde{\Phi}_{\nu}^{-1}$.
Hence, passing to the limit, we get 
$L_{\nu}\to_{\nu\to\infty} L_{\infty}$ of the form \eqref{operatoreFinale} with $Z_{\infty}$ given by
\eqref{tametame6}.
\end{proof}

 \section{Proof of  Theorem \ref{thm:main}}\label{sec5}
 In this short section we merge the two previous sections to prove 
 the reducibility of the Schr\"odinger equation \eqref{nls}: Theorem \ref{thm:main}.\\
We recall  that equation \eqref{nls} has the form 
\[
\pa_{t}u=-\ii (\Delta_{g}+\e W(\omega t))u
\]
where $\vphi\mapsto W(\vphi)$ is a $C^\infty$ map from $\T^d$ to $\A_\delta$, $\delta\leq 1/2$, and thus  $W\in\A_{\delta,r}$ for any $r>d/2$.
Its reducibility rely on the reducibility of the operator 
$\mathcal{F}$ in \eqref{starting-point}
with $V(\vphi)= \e W(\vphi)$. Rouglhy speaking we want to apply Theorem \ref{thm:regolo} to regularize $\mathcal F$ in such a way operator $\mathcal{F}$ is transformed 
into the operator $\mathcal{F}_{+}$
in \eqref{fine100}. Then we apply Lemma \ref{matrixest} to control the remainder $R$ in \eqref{fine100} in $s$-decay norm. This allows, for $\eps$ small enough, to apply the reducibility Theorem \ref{thm:redu} and to conclude.\\
 To justify all these steps we have to carefully follow the parameters and the smallness conditions. First we fix $\alpha\in (0,1)$ and $\gamma =\eps^\alpha$, $\delta\leq \frac12$, $s>n/2$ and $W$ belonging to all the $\A_{\delta,r}$ with $r>d/2$ . Then we fix $\rho$, $\mathtt b$, $\tau$ as in \eqref{costanteB}, we set $\ka=2\delta-1$ and  we fix $s_0>n/2$ and $S$ such that $s$ and $s+\mathtt b$ belong to $[s_0,S]$ and $S\geq p(\delta,0)$ (see \eqref{action}). Finally we set
$$\rho_0=S+\rho+\frac12,\quad r_0=S.$$
With these values of $\rho_0$, $r_0$, Theorem \ref{thm:regolo}  provide us with $\eps_*=\eps_*(S,n,d,\delta)$, $r_*=r_*(S,n,d,\delta)$ and $p=p(S,n,d,\delta)$ such that if $r>r_*$ and
\begin{equation}\label{yes}\gamma^{-1}\mathcal N_{\delta,r,p}(\eps W)<\eps_*(n,d,\delta)\end{equation}
then we can apply Theorem \ref{thm:regolo} to $\mathcal{F}$ with $V=\eps W$.  Since $W$ belongs to $\A_{\delta,r}$ for any $r>d/2$ , \eqref{yes} is satisfied for $\eps$ small enough.
So there exists $\Phi(\vphi)\in\L(H^s,H^s)$ such that (see \eqref{fine100})
$$\Phi \mathcal{F} \Phi^{-1} =\mathcal{F}_{+}=\omega\cdot\pa_{\vphi}+\ii(\Delta_g+Z+R)\,.$$
Further we knows that  $R\in \mathcal{R}_{\rho_0,r_0}^{\gamma,\mathcal{O}_0}$ and 
\begin{equation}\label{Rregul}
|R|_{\rho_0,r_0,s}^{\gamma,\mathcal{O}_0} \leq C
\Nc_{\delta,r,p}(\eps W) \,.
\end{equation}
Now we apply Lemma \ref{matrixest} to conclude that $R\in\mathcal M^{\g,\mathcal O_0}_{\rho,S}$ and 
\begin{equation}\label{Rregul1}
\bral R\brar_{\rho,S}^{\gamma,\mathcal{O}_0} \leq C
\Nc_{\delta,r,p}(\eps W) \,.
\end{equation}

We notice that the operator 
$\mathcal{F}_{+}$ has the same form of the operator $L_0$ in 
\eqref{start} with $Z_0= Z$, $R_0= R$ and 
$\mathcal{O}=\mathcal{O}_0$ (see \eqref{zeroMel}).
The remainder $R_0$ satisfies
the assumption $({\bf A2})$ by the discussion above.
Notice also that, in view of by  \eqref{Rregul1},
the constant $\epsilon$ given by \eqref{paramPiccolez} satisfies 
\begin{equation}\label{barnantes3}
\epsilon\leq \Nc_{\delta,r,p}( W)\e^{1-\alpha}
\end{equation}
and thus the smallness condition \eqref{piccolopiccolo} is satisfied provided that $\e$ is small enough.
We now prove that $Z_0$ satisfies assumption $({\bf A1})$ with $\kappa:=2\delta-1$.
First we note that, since $\delta\leq 1/2$ then $\kappa\leq 0$.
Moreover, by Theorem \ref{thm:regolo}, we have that
$Z_0:=Z=Z_1+Z_2$
with  $Z_{1}\in\mathcal{A}_{\delta}$ independent of $\omega\in \mathcal{O}_0$, and 
$Z_{2}\in\mathcal{A}_{2\delta-1}^{\gamma,\mathcal{O}_0}$. 
Estimate \eqref{stimefinali1} implies that for all $s\in[s_0,S]$
$$ \Nc_{\delta,s}(Z)\leq C
\Nc_{\delta,r,p}(\eps W)\leq C \Nc_{\delta,r,p}( W)\e^{1-\alpha}.
$$
Since $S\geq p(\delta,0)$ we deduce by \eqref{action} that  $$\|Z\|_{\L(L^2,H^{-\delta})}\leq C \Nc_{\delta,r,p}( W)\e^{1-\alpha}\leq \frac {c_0}2 $$ for $\e$ small enough  which in turn implies  that
$$ |\mu^{(0)}_{k,j}(\om)|\leq \frac {c_0}2 |k|^\delta
$$
and thus \eqref{mu1} holds true. Furthermore since $Z_1$ does not depend on $\om$, we have 
$$ \|(Z_0)_{[k]}^{[k]}\|_{\L(L^2)}^{lip,\O}=\|(Z_2)_{[k]}^{[k]}\|_{\L(L^2)}^{lip,\O} $$
 and thus  \eqref{stimefinali1} implies also \eqref{mu2} for $\e$ small enough.\\
Hence all the hypothesis of Theorem \ref{thm:redu} are satisfied for $L_0=\mathcal{F}_{+}$ and this theorem provides a
set of frequencies $\mathcal{O}_{\epsilon}$ such that, for $\omega\in \mathcal{O}_{\epsilon}$,
there is a 
 map $\Phi_{\infty}$ satisfying \eqref{stimaMappain}
such that $L_0\equiv\mathcal{F}_+$ transforms into $L_{\infty}$ in \eqref{operatoreFinale}.
By \eqref{misura} we have
\[
{\rm meas}(\mathcal{O}_0\setminus\mathcal{O}_{\epsilon})\leq C\gamma\,,
\]
for some constant $C>0$ depending on $s$.
It is also know that (recall \eqref{zeroMel}) 
${\rm meas}([1/2,3/2]^{d}\setminus\mathcal{O}_{0})\leq C\gamma$. Therefore, 
recalling that we set $\gamma=\e^\alpha$
we have that the \eqref{measure} holds. 
For $\omega\in \mathcal{O}_{\e}$
we set 
\[
\Psi(\omega t):=\Phi_{\infty}(\omega t)\circ \Phi(\omega t)\,.
\] 
By construction
the function $v:=\Psi(\omega t)u$ satisfies the equation \eqref{NLSred}
with $\e Z\rightsquigarrow Z_{\infty}$ in \eqref{operatoreFinale}.
Moreover, by 
\eqref{stimefinali3}, \eqref{stimefinali4}, \eqref{stimaMappain} and \eqref{Rregul}, we have
\begin{equation}\label{tripoli5}
\begin{aligned}
\sup_{\vphi\in\T^d}\|\Psi^{\pm1}(\vphi)-{\rm Id}\|_{\L(H^s,H^{s-\delta})}
& \leq  \gamma^{-1}C_{s}\e \,,\\
\sup_{\vphi\in\T^d}\|\Psi^{\pm1}(\vphi)\|_{\L(H^s,H^{s})} &\leq 
1+\gamma^{-1}C_{s}\e\,,
\end{aligned}
\end{equation}
for some $C_{s}>0$. 
This concludes the proof.

\appendix
\section{Technical lemmata}
\subsection{Proof of Lemma \ref{propertiespdo}}\label{appendixAA}
\begin{proof}[{\bf Proof of Lemma \ref{propertiespdo}}]
The bounds \eqref{algSeminorm}, \eqref{algSeminorm2} and \eqref{algSeminorm3}
can be deduced by using the properties of the  
semi-norm  in 
\eqref{action}-\eqref{commuPseudo} and the definition
in \eqref{psdo}.
We give the proof  the bound \eqref{estOmega-1} of item $(iii)$. 
The bound \eqref{estOmega} is similar.
We have that 
 \begin{equation}\label{cinest3}
  B(\omega):=(\omega \cdot\partial_{\vphi})^{-1}A(\omega) 
  = \sum_{0 \neq l \in \mathbb{Z}^d} 
  \frac{1}{\ii \omega \cdot l}e^{il\cdot \vphi} A(\omega;l).
  \end{equation}
  Thus
  \begin{equation}\label{cinest}
      \begin{aligned}
 \big( \Nc_{ m,r-\alpha,p} (&(\omega\cdot \partial_{\vphi})^{-1}A)\big)^{2}
 \stackrel{\eqref{psdo}}{\leq}
 \sum_{0 \neq l\in \mathbb{Z}^{d}} 
 \frac{1}{| \omega \cdot l|^{2}} \langle l \rangle^{2(r-\alpha)} 
 \mathcal{N}_{m,p}^2 (A(l ))
    \\&
   \stackrel{\eqref{diodio}}{\leq} \frac{1}{\gamma^{2}} \sum_{ l\in \mathbb{Z}^{d}\setminus\{0\}} \langle l \rangle^{2(r-\alpha)}
 | l|^{2\alpha}  \mathcal{N}^2_{m,p} (A(l ))
  \\&\le  \frac{C}{\gamma^{2} } \sum_{0 \neq l
  \in \mathbb{Z}^{d}} \langle l \rangle^{2r}
\mathcal{N}^2_{m,p} (A(l ))=\frac{C}{\gamma^{2}}\big(\mathcal{N}_{m,r,p}(A)\big)^{2}\,.
\end{aligned}
\end{equation}
To estimate $\mathcal{N}^{lip,\mathcal{O}}_{m,r-(2\alpha+1),p}(B)$ (see \eqref{suplip})
we reason as follow.
We first note that
\[
\begin{aligned}
B(\omega_1)-B(\omega_2)&=
\sum_{0 \neq l \in \mathbb{Z}^d}
  \frac{1}{\ii \omega_1 \cdot l}e^{il\cdot \vphi} \Big(A(\omega_1;l)-A(\omega_{2};l)\Big)\\
  &+\sum_{0 \neq l \in \mathbb{Z}^d} 
\frac{(\omega_{1}-\omega_2)\cdot l}{\ii (\omega_1\cdot l)(\omega_{2}\cdot l)}e^{il\cdot \vphi} A(\omega_2;l).
\end{aligned}
\]
Moreover, by using \eqref{diodio} and that $\mathcal{O}$ is compact, we have
\[
\left|\frac{(\omega_{1}-\omega_2)\cdot l}{\ii (\omega_1\cdot l)(\omega_{2}\cdot l)}\right|\leq
C \frac{1}{\gamma^{2}}|l|^{2\alpha+1}|\omega_1-\omega_2|\,.
\]
Therefore reasoning as in \eqref{cinest} we get
\begin{equation}\label{cinest2}
\frac{\mathcal{N}_{m,r-(2\alpha+1),p}(B(\omega_1)-B(\omega_2))}{|\omega_1-\omega_2|}\lesssim 
\frac{1}{\gamma}\mathcal{N}_{m,r,p}^{lip,\mathcal{O}}(A)+\frac{1}{\gamma^{2}}\mathcal{N}^{sup,\mathcal{O}}_{m,r,p}(A)\,.
\end{equation}
Combining \eqref{cinest}, \eqref{cinest2} and recalling \eqref{Lip-norm}, \eqref{cinest3}
we obtained
\begin{equation*}
\begin{aligned}
\mathcal{N}^{\gamma,\mathcal{O}}_{m,r-(2\alpha+1),p}(B)&\lesssim 
\frac{1}{\gamma}\mathcal{N}_{m,r,p}^{sup,\mathcal{O}}(A)+
\gamma\left( 
\mathcal{N}_{m,r,p}^{lip,\mathcal{O}}(A)+\frac{1}{\gamma^{2}}\mathcal{N}_{m,r,p}^{sup,\mathcal{O}}(A)
\right)\\
&\lesssim \frac{1}{\gamma}\left(
\mathcal{N}_{m,r,p}^{sup,\mathcal{O}}(A)+
\gamma\mathcal{N}_{m,r,p}^{lip,\mathcal{O}}(A)\right)
\end{aligned}
\end{equation*}
which is bound \eqref{estOmega-1}.
\end{proof}

\subsection{Properties of the $s$-decay norm}\label{techtech}

In this appendix $s_0$ is some fixed number satisfying $s_0> (d+n)/2$.
\begin{lemma}\label{DecayAlg}
Let $\al>0$. 
Then (recall \eqref{Diag}, \eqref{decayNorm2}) 
\begin{equation}\label{beebee}
\bral \mathcal{D}^{\pm\al} A \mathcal{D}^{\mp\al}\brar^{\gamma,\mathcal{O}}_{s}
\le_{s}
\bral A\brar^{\gamma,\mathcal{O}}_{s+\al}\,,
\end{equation}
\begin{equation}\label{beebee2}
\bral\mathcal{D}^{\pm\al} (\Pi_{N}A) \mathcal{D}^{\mp\al}\brar^{\gamma,\mathcal{O}}_{s}
\le_{s}N^{\alpha}
\bral A\brar^{\gamma,\mathcal{O}}_{s}\,.
\end{equation}
\end{lemma}

\begin{proof}
The bounds \eqref{beebee}, \eqref{beebee2} follow
by reasoning as in the proof of Lemma $A.1$ in \cite{FG1}
and using the \eqref{cutOff}.
\end{proof}

\begin{lemma}\label{decayspaziotempo}
Let $A$ be a matrix as in \eqref{timeMat}
with finite $\bral \cdot\brar_{s}$-norm (see \eqref{decayNorm2}).
 Then (recall \eqref{decayNorm1}) one has
\[
\|A(\vphi)\|_{\L(H^s,H^s)}\leq_s |A(\vphi)|_{s}
 \leq_{s}
 \bral A\brar_{s+s_0}\,,\quad \forall\,\vphi\in \mathbb{T}^{d}\,.
\]
\end{lemma}
\begin{proof}
See Lemma $2.4$ in \cite{BBM14}.
\end{proof}

\begin{lemma}\label{well-well}
Assume that
\begin{equation}\label{mito30}
C(s) \bral A\brar_{s_0}^{\gamma,\mathcal{O}}\leq 1/2
\end{equation}
for some large $C(s)>0$ depending on $s\geq s_0$.
Then the map $\Phi:={\rm Id}+\Psi$ defined as
\begin{equation}\label{exp}
\Phi:=e^{\ii A}:=\sum_{p\geq0}\frac{1}{p!}(\ii A)^{p}\,,
\end{equation}
satisfies
\begin{equation}\label{brackettech}
\bral \Psi\brar_{s}^{\gamma,\mathcal{O}}
\leq_{s} \bral A\brar_{s}^{\gamma,\mathcal{O}}\,,
\end{equation}
\end{lemma}
\begin{proof}
For any $n\geq1$, using \eqref{decayTame}, we have, for some $C(s)>0$,
\[
\begin{aligned}
\bral A^{n}\brar_{s_0}&\leq [C(s_0)]^{n-1}\bral A\brar^{n}_{s_0}\,,\\
\bral A^{n}\brar_{s}&\leq n[C(s) \bral A\brar_{s_0}]^{n-1}
C(s) \bral A\brar_{s}\,,\quad \forall\, s\geq s_0\,.
\end{aligned}
\]
The same holds also for the norm $\bral\cdot\brar_{s}^{\gamma,\mathcal{O}}$
Hence
\[
\bral \Psi\brar_{s}^{\gamma,\mathcal{O}}
\leq
\bral A\brar_{s}^{\gamma,\mathcal{O}}
\sum_{p\geq 1}\frac{C(s)^{p}}{p!}
(\bral A\brar_{s_0}^{\gamma,\mathcal{O}})^{p-1}\,,
\]
for some (large) $C(s)>0$. By the smallness condition \eqref{mito30}
one deduces   the bounds \eqref{brackettech}.
\end{proof}

\begin{lemma}\label{DecayAlg2}
Let $\alpha,\beta\in \mathbb{R}$.
Then
\begin{align}
&\bral A M\brar_{\alpha+\beta,s}^{\gamma,\mathcal{O}}
\le_{s} \bral A\brar^{\gamma,\mathcal{O}}_{\alpha,s+|\beta|}
\bral M\brar^{\gamma,\mathcal{O}}_{\beta,s_0+|\alpha|}
+
\bral A\brar^{\gamma,\mathcal{O}}_{\alpha,s_0+|\beta|}
\bral M\brar^{\gamma,\mathcal{O}}_{\beta,s+|\alpha|}
\,,\label{stim2}\\
&\bral({\rm Id}- \Pi_{N})M\brar^{\gamma,\mathcal{O}}_{\beta,s}\le_{s} 
N^{-\mathfrak{s}}
\bral M\brar^{\gamma,\mathcal{O}}_{\beta,s+\mathfrak{s}}\,,
\qquad \mathfrak{s}\geq 0\,.\label{stim3}
\end{align}
Moreover, if $\alpha\leq\beta< 0$ then
\begin{equation}\label{stim1}
\bral AM \brar^{\gamma,\mathcal{O}}_{\beta,s}
\le_{s}\bral A\brar^{\gamma,\mathcal{O}}_{\alpha,s}
\bral M\brar^{\gamma,\mathcal{O}}_{\beta,s_0}+
\bral A\brar^{\gamma,\mathcal{O}}_{\alpha,s_0}
\bral M\brar^{\gamma,\mathcal{O}}_{\beta,s}\,.
\end{equation}
\end{lemma}

\begin{proof}
To prove \eqref{stim2} 
one reasons as in Lemma $A.3$ in \cite{FG1}.
The \eqref{stim1} and \eqref{stim3} follow by Lemma \ref{DecayAlg}.
\end{proof}

 \begin{lemma}\label{AzioneDec}
One has
\begin{equation}\label{mito}
\|\mathcal{D}^{\beta} A h\|_{\ell_{s}}\le_{s}
\bral A\brar_{\beta,s}
\|h\|_{\ell_{s_0}}+
\bral A\brar_{\beta,s_0}
\|h\|_{\ell_{s}}\,,
\end{equation}
for any $h\in\ell_{s}$ (see \eqref{elleNorma}) and $\beta\in \mathbb{R}$.
\end{lemma}
\begin{proof}
One reasons as in Lemma $A.4$ in \cite{FG1}.
\end{proof}

\subsection{Flows of pseudo differential operators}\label{AppC}

\begin{lemma}\label{lemmaB1}
Fix $m\leq0$, $0\leq \delta\leq1$, $r>d/2$ and $\rho\geq0$
and consider $S_1\in\mathcal{A}_{m,r}^{\gamma,\mathcal{O}}$ and
$S_{2}\in \mathcal{A}_{\delta,r}^{\gamma,\mathcal{O}}$  
(see Definition \ref{def:pseudotempo}).
Assume also that 
\begin{equation}\label{lemmaB111}
[S_2,K_0]=0\,, \quad \langle S_2 h, v\rangle=\langle h, S_2v\rangle
\end{equation}
where $\langle \cdot,\cdot\rangle$ is the standard $L^{2}$ scalar product.
Let us define
\begin{equation}\label{flows}
\Phi_1^{\tau}:=\Phi_1^{\tau}(\vphi):=e^{\tau \ii S_1}\,, \qquad 
\Phi_2^{\tau}:=\Phi_2^{\tau}(\vphi):=e^{\tau \ii S_2}\,.
\end{equation}
For any $s\geq 0$ 
there are $\e_0, C,p>0$ such that,
for any $0<\e\leq \e_0$, 
 if
\begin{equation}\label{smallcondition}
\mathcal{N}^{\gamma,\mathcal{O}}_{m,r,p}(S_1)
+\mathcal{N}^{\gamma,\mathcal{O}}_{\delta,r,p}(S_2)\leq \e,
\end{equation}
then the following holds true:

\noindent
$(i)$ the map $\Phi_{1}^{\tau}$ satisfies 
\begin{align}
\sup_{\vphi\in\mathbb{T}^{d}}\|\Phi_1^{\tau}(\vphi)
-{\rm Id}\|^{\gamma,\mathcal{O}}_{\mathcal{L}(H^{s};H^{s-m})}
&\leq
C \mathcal{N}^{\gamma,\mathcal{O}}_{m,r,p}(S_1)\,,\label{flusso11}\\
\sup_{\vphi\in \mathbb{T}^{d}}\|(\pa^{k}_{\vphi}\Phi_1^{\tau})(\vphi)\|^{\gamma,\mathcal{O}}_{\L(H^{s},H^s)}
&\leq  
C \mathcal{N}^{\gamma,\mathcal{O}}_{\delta,r,p}(S_1)\,,\quad 0\leq k\leq r\,,\label{flusso12}
\end{align}
for any $\tau\in [0,1]$;

\noindent
$(ii)$ the map $\Phi_{2}^{\tau}$ satisfies 
\begin{align}
\sup_{\vphi\in \mathbb{T}^{d}}\|\Phi_2^{\tau}(\vphi)\|_{\L(H^{s},H^s)}&\leq   
(1+C \mathcal{N}^{\gamma,\mathcal{O}}_{\delta,r,p}(S_2))\,,\label{flusso1}\\
\sup_{\vphi\in \mathbb{T}^{d}}\|(\Phi_2^{\tau}(\vphi)-{\rm Id})\|_{\L(H^{s},H^{s-\delta})}
&\leq 
C \mathcal{N}^{\gamma,\mathcal{O}}_{\delta,r,p}(S_2)\,,\label{flusso2}\\
\sup_{\vphi\in \mathbb{T}^{d}}\|(\pa^{k}_{\vphi}\Phi_2^{\tau})(\vphi)\|_{\L(H^{s},H^{s-k\delta})}
&\leq 
C \mathcal{N}^{\gamma,\mathcal{O}}_{\delta,r,p}(S_2)\,,\quad 1\leq k\leq r\,,
\label{flusso3}
\end{align}
for any $\tau\in [0,1]$ and any $\omega\in\mathcal{O}$.
Moreover the following bounds on the Lipschitz norm hold true:
\begin{align}
\sup_{\vphi\in \mathbb{T}^{d}}\|\Phi_2^{\tau}(\vphi)\|^{\gamma,\mathcal{O}}_{\L(H^{s},H^{s-1})}&\leq   
(1+C \mathcal{N}^{\gamma,\mathcal{O}}_{\delta,r,p}(S_2))\,,\label{flusso1bis}\\
\sup_{\vphi\in \mathbb{T}^{d}}\|(\Phi_2^{\tau}(\vphi)-{\rm Id})\|^{\gamma,\mathcal{O}}_{\L(H^{s},H^{s-\delta-1})}
&\leq 
C \mathcal{N}^{\gamma,\mathcal{O}}_{\delta,r,p}(S_2)\,,\label{flusso2bis}\\
\sup_{\vphi\in \mathbb{T}^{d}}
\|(\pa^{k}_{\vphi}\Phi_2^{\tau})(\vphi)\|^{\gamma,\mathcal{O}}_{\L(H^{s},H^{s-k\delta-1})}
&\leq 
C \mathcal{N}^{\gamma,\mathcal{O}}_{\delta,r,p}(S_2)\,,\quad 1\leq k\leq r\,,
\label{flusso3bis}
\end{align}
for   any $\tau\in [0,1]$.
\end{lemma}

\begin{proof}
We shall prove the result for the map $\Phi_{2}^{\tau}$. The estimates on $\Phi_{1}^{\tau}$
can be obtained in the same way.
Notice that the operator $\Phi_{2}^{\tau}$ solves the problem
\begin{equation}\label{Cau10}
\left\{\begin{aligned}
&\pa_{\tau}\Phi_{2}^{\tau}(\vphi)=\ii S_{2}(\vphi)\Phi_{2}^{\tau}(\vphi)\\
&\Phi_{2}^{0}(\vphi)={\rm Id}\,.
\end{aligned}\right.
\end{equation}
Using \eqref{Cau10} and the assumption \eqref{lemmaB111}
one can check that
\[
\pa_{\tau}\|\Phi_{2}^{\tau}h\|^{2}_{H^{s}}=0\qquad
\Rightarrow \qquad \|\Phi_{2}^{\tau}h\|_{H^{s}}\leq \|h\|_{H^{s}}\,,
\]
for any $\tau\in [0,1]$, $h\in H^{s}$ and $\vphi\in \mathbb{T}^{d}$. This is the \eqref{flusso1}.
Let us now define 
\[
\Gamma^{\tau}(\vphi):=\Phi_{2}^{\tau}(\vphi)-{\rm Id}\,.
\]
It solve the problem
\[
\pa_{\tau}\Gamma^{\tau}(\vphi)=\ii S_{2}(\vphi)\Gamma^{\tau}(\vphi)+\ii S_{2}(\vphi)\,,\qquad
\Gamma^{0}(\vphi)=0\,.
\]
By Duhamel formula have
\[
\Gamma^{\tau}(\vphi)h=\int_{0}^{\tau}\Phi_{2}^{\tau}\Phi_{2}^{-\s}(\vphi)\ii S_{2}(\vphi)hd\s\,.
\]
Therefore the bound \eqref{flusso2} follows by \eqref{flusso1} and the estimates on $S_2$.
Similarly the operator $(\pa_{\vphi}\Phi_{2}^{\tau})(\vphi)$
satisfies 
\begin{equation}\label{flusso4}
\left\{
\begin{aligned}
&\pa_{\tau}(\pa_{\vphi}\Phi_{2}^{\tau})(\vphi)=\ii S_{2}(\vphi)(\pa_{\vphi}\Phi_{2}^{\tau})(\vphi)
+\ii (\pa_{\vphi}S_2)(\vphi)\Phi_{2}^{\tau}(\vphi)\,,\\
&(\pa_{\vphi}\Phi_{2}^{0})(\vphi)=0\,.
\end{aligned}\right.
\end{equation}
We have that
\[
\sup_{\vphi\in \mathbb{T}^{d}}\|(\pa_{\vphi}S_2)(\vphi)\Phi_{2}^{\tau}(\vphi)h\|_{H^{s-\delta}}
\lesssim \|h\|_{H^{s}}
\mathcal{N}^{\gamma,\mathcal{O}}_{\delta,r,p}(S_2)
\]
by \eqref{flusso1} and the fact that $r>d/2$. Hence, using Duhamel formula and the \eqref{flusso1}, we deduce the \eqref{flusso3} for $k=1$. The \eqref{flusso3}
for $k>1$ can be obtained in the same way by differentiating \eqref{flusso4}.
The Lipschitz bounds \eqref{flusso1bis}-\eqref{flusso3bis} 
follows reasoning as in the estimates   of $\pa_{\vphi}\Phi_{2}^{\tau}(\vphi)$.
The bounds \eqref{flusso11}, \eqref{flusso12}
can be deduced reasoning as done above and using
the fact that the generator $\ii S_1(\vphi)$ is a \emph{bounded} pseudodifferential operator. 
\end{proof}

\begin{lemma}\label{lemmaB2}
Let $r_1\geq 0$ and $r>r_1+d/2$, $\delta>0$, $\rho_1>0$, $\rho:=\rho_1+\delta r_1+1$ and
consider  
 $R\in \mathcal{R}_{\rho,r}^{\gamma,\mathcal{O}}$ 
(see Definition \ref{Rrho}). 
Consider also the map $\Phi_{2}(\vphi):=\Phi_2^{\tau}(\vphi)_{|\tau=1}$,
where $\Phi_{2}^{\tau}(\vphi)$ is  given in Lemma \ref{lemmaB1}.
Then 
 $G_2(\vphi):=\Phi_{2}(\vphi) R(\vphi) \Phi_{2}^{-1}(\vphi)$ belongs to 
$\mathcal{R}^{\gamma,\mathcal{O}}_{\rho_1,r_1}$.
Moreover for any $s\geq0$ there exist $p$ and $C$ such that
\begin{align}
|G_2|^{\gamma,\mathcal{O}}_{\rho_1,r_1,s}
&\leq  |R|^{\gamma,\mathcal{O}}_{\rho,r,s}
(1+C\mathcal{N}_{\delta,r,p}^{\gamma,\mathcal{O}}(S_2))\,.\label{ubdd1}
\end{align}
\end{lemma}

\begin{proof}
We need to prove  that the map
$\vphi\mapsto \Gamma(\vphi)$
is in $H^{r_1}(\mathbb{T}^{d}; \mathcal{L}(H^{s};H^{s+\rho_1}))$. 
We note that 
\begin{align}
|G_2&|_{\rho_1,r_1,s}\lesssim\sum_{k=0}^{r_1}\sup_{\vphi\in\mathbb{T}^{d}}
\|(\pa_{\vphi}^{k}G_2)(\vphi)\|_{\mathcal{L}(H^{s};H^{s+\rho_1})}\label{flusso52}\\
&\lesssim
\sum_{k=0}^{r_1}
\sum_{\substack{ k_1+k_2+k_3=k\\ k_i\geq 0}}
\sup_{\vphi\in\mathbb{T}^{d}}\|(\pa_{\vphi}^{k_1}\Phi_2)(\vphi)(\pa_{\vphi}^{k_2}R)(\vphi)
(\pa_{\vphi}^{k_3}\Phi^{-1}_2)(\vphi)\|_{{\mathcal{L}(H^{s};H^{s+\rho_1})}}\,.\nonumber
\end{align}
We estimate separately each summand in \eqref{flusso52}.
First of all notice that,
by the definition of the norm in \eqref{flusso50} and the fact that $r>r_1+d/2$,
one has
\begin{equation}\label{flusso51}
\sup_{\vphi\in\mathbb{T}^{d}}\|\pa_{\vphi}^{k_{2}}(R(\vphi))\|_{\mathcal{L}(H^{s};H^{s+\rho})}\lesssim
|R|_{\rho,r,s}\,.
\end{equation}
Hence the summand in \eqref{flusso52} 
with $k_1=k_3=0$ is trivially bounded by the right hand side
in \eqref{ubdd1}.
If at least one between $k_1,k_2$  is different from zero we have, for any $h\in \mathbb{H}^{s}$,
\[
\begin{aligned}
\|(\pa_{\vphi}^{k_1}\Phi_2)(\vphi)&(\pa_{\vphi}^{k_2}R)(\vphi)
(\pa_{\vphi}^{k_3}\Phi^{-1}_2)(\vphi)h\|_{H^{s+\rho_1}}\\
&\stackrel{\eqref{flusso3}}{\lesssim}
\mathcal{N}_{\delta,r,p}^{\gamma,\mathcal{O}}(S_2)
\|(\pa_{\vphi}^{k_2}R)(\vphi)
(\pa_{\vphi}^{k_3}\Phi^{-1}_2)(\vphi)h\|_{H^{s+\rho_1+k_1\delta}}\\
&\stackrel{\eqref{flusso51}}{\lesssim}\mathcal{N}_{\delta,r,p}^{\gamma,\mathcal{O}}(S_2)|R|_{\rho,r,s}
\|(\pa_{\vphi}^{k_3}\Phi^{-1}_2)(\vphi)h\|_{H^{s+\rho_1+k_1\delta-\rho}}\\
&\stackrel{\eqref{smallcondition}}{\lesssim}
\mathcal{N}_{\delta,r,p}^{\gamma,\mathcal{O}}(S_2)|R|_{\rho,r,s}
\|h\|_{H^{s+\rho_1+(k_1+k_3)\delta-\rho}}\,.
\end{aligned}
\]
Notice that $+\rho_1+(k_1+k_3)\delta-\rho\leq 0$ since $k_1+k_3\leq r_1$ and that
the estimate above is uniform in $\vphi\in \mathbb{T}^{d}$.
Hence, together with the \eqref{flusso52}, it implies the \eqref{ubdd1}
for the norm $|\cdot|_{\rho_1,r_1,s}$.
The Lipschitz bounds are obtained similarly taking into account the extra loss of derivatives
appearing in the estimates \eqref{flusso1bis}-\eqref{flusso3bis}.
\end{proof}
Similarly we prove in the bounded case:
\begin{lemma}\label{lemmaB3}
Let $r_1\geq 0$ and $r>r_1+d/2$,   $\rho>0$ and
consider  
 $R\in \mathcal{R}_{\rho,r}^{\gamma,\mathcal{O}}$.
Consider also the map $\Phi_{1}(\vphi):=\Phi_1^{\tau}(\vphi)_{|\tau=1}$,
where $\Phi_{1}^{\tau}(\vphi)$ is  given in Lemma \ref{lemmaB1}.
Then 
 $G_1(\vphi):=\Phi_{1}(\vphi) R(\vphi) \Phi_{1}^{-1}(\vphi)$ belongs to 
$\mathcal{R}^{\gamma,\mathcal{O}}_{\rho,r}$.
Moreover for any $s\geq0$ there exist $p$ and $C$ such that
\begin{align}
|G_1|_{\rho,r_1,s}^{\gamma,\mathcal{O}}
&\leq  C|R|_{\rho,r,s}^{\gamma,\mathcal{O}}
(1+\mathcal{N}_{m,r,p}^{\gamma,\mathcal{O}}(S))\,,\label{ubdd2}
\end{align}
\end{lemma}
%

\end{document}